\documentclass{amsart}

\usepackage{lmodern}
\usepackage{amsmath,amssymb,amscd}
\usepackage{hyphenat}
\usepackage[utf8]{inputenc}
\usepackage[OT1]{fontenc}
\usepackage{stmaryrd}
\usepackage{MnSymbol}
\usepackage{color}

\RequirePackage[colorlinks=true, breaklinks=true, urlcolor= black, linkcolor= black, citecolor= black, bookmarksopen=true,linktocpage=true,plainpages=false,pdfpagelabels]{hyperref}
\usepackage[capitalize]{cleveref}
\usepackage{tikz-cd}

\renewcommand{\mid}{:}

\newcommand\Qq{\mathbb{Q}}
\newcommand\Rr{\mathbb{R}}
\newcommand\Zz{\mathbb{Z}}
\newcommand{\fm}{{\mathfrak m}} 
\newcommand{\cO}{{\mathcal O}}
\renewcommand{\projlim}{\varprojlim}
\newcommand{\rD}{\mathrm{D}}
\newcommand{\W}{\mathrm{W}}
\newcommand{\w}{\mathrm{w}}
\newcommand\fU{\mathfrak{U}}

\renewcommand\L{\mathfrak{L}}
\newcommand{\Lrg}{\L_\mathrm{rg}}
\newcommand{\LD}{\L_\rD}

\newcommand{\Ldiv}{\L_{\mathrm{div}}}
\newcommand{\Ldivpi}{\L_{\mathrm{div},\varpi}}
\newcommand{\Val}{\mathcal{O}}

\newcommand{\MVF}{\mathrm{MVF}}
\newcommand{\PERF}{\mathrm{PERF}}
\newcommand{\ACMVF}{\mathrm{ACMVF}}
\newcommand{\ACVF}{\mathrm{ACVF}}
\newcommand{\TVR}{\mathrm{TVR}}
\newcommand{\TPERF}{\mathrm{TPERF}}
\newcommand\id{\mathrm{id}}
\newcommand\TP{\mathrm{S}}
\newcommand\Cont{\mathrm{Cont}}
\newcommand\res{\mathrm{res}}
\renewcommand\div{\mathop{|}}
\newcommand\alg[1]{#1^{\mathrm{alg}}}
\newcommand\Spec{\mathrm{Spec}}
\newcommand\Spa{\mathrm{Spa}}

\newtheorem{thm}{Theorem}[section]
\newtheorem{cor}[thm]{Corollary}
\newtheorem{claim}[thm]{Claim}

\newtheorem{prop}[thm]{Proposition}
\newtheorem{lem}[thm]{Lemma}
\newtheorem{fact}[thm]{Fact}

\theoremstyle{definition}
\newtheorem{defn}[thm]{Definition}
\newtheorem{rem}[thm]{Remark}

\begin{document}

\title{The tilting equivalence as a
bi-interpretation}
\author{Silvain Rideau-Kikuchi}
\address{\'{E}cole Normale Sup\'{e}rieure \\
CNRS\\
D\'{e}partement de math\'{e}matiques et applications \\
45 rue d'Ulm \\
75230 PARIS cedex 05
\\ 
France}
\email{silvain.rideau-kikuchi@ens.fr}

\author{Thomas Scanlon}
\address{University of California, Berkeley \\
Department of Mathematics \\
Evans Hall \\
Berkeley, CA 94720-3840 \\
USA}
\email{scanlon@math.berkeley.edu}

\author{Pierre Simon}
\email{pierre.sim85@gmail.com}

\begin{abstract}
We propose a model theoretic interpretation
of the theorems about the equivalence between 
mixed characteristic perfectoid spaces and 
their tilts.
\end{abstract}

\maketitle

\section{Introduction}
A perfectoid field is a complete non-trivially normed field \((K,\vert\cdot\vert)\)
of residue characteristic \(p\) for some prime number \(p\) for 
which the set of norms \(|K^\times|\) is dense in \(\mathbb{R}_+\) and 
the map \(x \mapsto x^p\) is surjective as a self-map 
on the ring \(\mathcal{O}(K)/p \mathcal{O}(K)\) where 
\(\mathcal{O}(K) = \{x \in K : \vert x \vert \leq 1 \}\) is the ring of 
integers of \(K\).  Given a perfectoid field \(K\), the tilt \(K^\flat\) is 
obtained, as a multiplicative monoid, as the projective limit of 
the system in which \(K\) maps to itself repeatedly via 
\(x \mapsto x^p\): \[K^\flat := 
\varprojlim ( \begin{tikzcd} \cdots \ar[r] & K \ar[r,"{x \mapsto x^p}"] & 
 K \ar[r,"{x \mapsto x^p}"] & K \end{tikzcd})\]
 The tilt \(K^\flat\) carries a natural structure of a complete normed 
 field of characteristic \(p\).  (We recall the details of this 
 construction in Section~\ref{sec:bi-interpretations}.)  
It can happen that two different perfectoid fields have isomorphic tilts. For a
 fixed  complete perfect nontrivially normed field \(L\) of characteristic \(p\) the family
 of untilts, that is, characteristic zero perfectoid fields \(K\) with \(K^\flat
 \cong L\), is parameterized by the Fargues\hyp{}Fontaine curve \cite{FarFon}.
 That is, there is a natural correspondence between perfectoid fields and
 perfect non-trivially normed fields fields of positive characteristic given with
 a point on the Fargues\hyp{}Fontaine curve.

 This tilting correspondence has been used to 
 transfer properties between mixed and 
 positive characteristic.  For example, 
 in~\cite{Scholze}, Scholze extends the 
 tilting operation to adic spaces, establishes
 an equivalence of categories of adic spaces
 over \(K\) and its tilt \(K^\flat\), and 
 uses this equivalence to transfer the truth 
 of the weight monodromy conjecture 
 for complete intersections in positive
 characteristic to mixed characteristic.  For another
 celebrated use of 
 perfectoids, consider 
 Andr\'{e}'s proof of
 the 
 direct summand conjecture
 in~\cite{Andre-facteur-direct} in which he ports positive characteristic
 proof techniques to mixed 
 characteristic.   
 
 Unlike the case of the 
 Ax-Kochen-Ershov-type theorems
 in which theories of henselian 
 fields of mixed 
 characteristic and of positive 
 characteristic converge 
 as the residue characteristic 
 grows, the tilt/untilt 
 correspondence is not 
 asymptotic; it allows for 
 direct comparisons between 
 mixed characteristic perfectoid
 fields and positive characteristic
 perfect fields with exactly the 
 same residue fields.  This is 
 curious from the 
standpoint of mathematical logic as 
the correspondence could not possibly reflect 
an equality of theories, since, for 
instance, the characteristic of the 
field is captured by the theory.  
This article is motivated by the 
problem of answering the riddle of 
how the tilt/untilt correspondence
might be explained through 
mathematical logic.

An answer is given by
Jahnke and Kartas in~\cite{JahKar-Perf}. Their main theorem is that if \(K\) is
a perfectoid field, then for any non-principal ultrapower \(K^\fU\), the tilt
\(K^\flat\) of \(K\) may be realized as an elementary structure of a residue
field of \(K^\fU\). (See~\cite[Theorem 6.2.3]{JahKar-Perf} for details.)  Their
proof gives an equivalence of categories between the category of finite
\'{e}tale algebras over \(K\) and the corresponding category over \(K^\flat\)
and the theorem itself allows them to replace parts of almost mathematics~\cite{Gabber-Ramero}
with ordinary commutative algebra.

Our answer is that the tilt/untilt correspondence is a quantifier\hyp{}free
bi\hyp{}interpretation in continuous logic, largely based on 
a reformulation of Fargues and Fontaine's
account of the tilting equivalence \cite{FarFon}. We note in
Section~\ref{sec:consequences} that the existence of this bi\hyp{}interpretation
implies that various important features of perfectoid fields are preserved by
the tilt/untilt correspondence.  These consequences include the
Fontaine\hyp{}Wintenberger theorem that a perfectoid field and its tilt have
canonically isomorphic Galois groups, identifications between adic spaces over a
perfectoid field and adic space over its tilt, and an approximation lemma
whereby the tilt of an algebraic variety over a perfectoid field may be
approximated by sets defined by quantifier\hyp{}free first\hyp{}order formulas
in the language of valued fields.   

This present paper is long in the tooth. The main observations and constructions were already 
completed while the three authors were working together at Berkeley during the 2016/7
academic year.   Our delay in 
preparing the work for publication resulted from our thwarted attempts to upgrade the approximation lemma
to a strong enough form to allow for a transfer of further theorems in positive characteristic to 
mixed characteristic, the weight monodromy conjecture without a restriction to complete intersections being our
target application.  We doubt that our methods will suffice to achieve such an end, but, perhaps, 
a reader cleverer than ourselves could implement the strategy.  

For us, the main contribution of the
present paper is to demonstrate that the theory 
and important 
consequences of the tilt/untilt correspondence may be established
through the conceptual lens of the continuous logic 
of valued fields requiring only a minimum of direct computation.  
Of course, mathematics of this depth requires some real technical work and this appears on the surface here 
with the constructions taken from~\cite{FarFon} and in the background with the modern versions (e.g.~\cite{HHM}) of the 
theorems of Abraham Robinson on the theory of algebraically closed valued fields~\cite{Robinson}.

Our work leaves open some 
natural problems about 
understanding the theory
of perfectoids through 
continuous logic.  
We have already noted the 
aim of obtaining more
general approximation 
results from the 
bi\hyp{}interpretation.
Another goal would be to 
explain the tilt/untilt
correspondence for more 
general perfectoid 
algebras as a 
bi\hyp{}interpretation.
It seems to us that most
of the explicit constructions
we carry out work 
at a higher level of 
generality.  A technical 
problem is that we would 
need a better continuous
logic theory of \(p\)-adic
Banach algebras to express
the tilt/untilt 
correspondence as
our expected 
bi\hyp{}interpretation.  
This issue is related to 
a point we make with 
Remark~\ref{rem:BY-MVF}
that with Ben Yaacov's 
formulation of the 
theory of metric valued 
fields it is not obvious
that untilting can be 
made into an interpetation 
because, relative to that 
theory, valuation rings are
not definable.  Finally, 
as we note in Remark~\ref{rem:correct-topology}, the equivalence
(see Section~\ref{sec:type-spaces})
we may establish between 
adic spaces using 
continuous logic 
type spaces does not
give the right topology and  
we expect that a resolution 
of this issue would follow 
from a presentation of the 
tilt/untilt correspondence
based on \emph{positive}
continuous logic. 

This paper is organized as follows.  In Section~\ref{sec:some-contunuous-logic} we recall
the formalism of continuous logic for metric structures and work out the basics of interpretations in 
continuous logic.  In Section~\ref{sec:MVF} we introduce and develop 
the theories of metric valued fields and, 
more specifically, of perfectoid fields from the 
standpoint of continuous logic.  Section~\ref{sec:bi-interpretations} comprises the core this paper.  
There we construct bi\hyp{}interpretations between three theories: perfectoid fields, truncated perfectoid rings, 
and perfect metric valued fields of positive characteristic equipped with a Fargues\hyp{}Fontaine 
parameter.   In Section~\ref{sec:consequences} we express some consequences of the existence
of the bi\hyp{}interpretations 
produced in the earlier 
section.

\subsection*{Acknowledgements}

During the writing of this paper, S. R.-K. was partially supported by grants ValCoMo (ANR-13-BS01-0006) and GeoMod AAPG2019 (ANR-DFG).
T.S. was partially supported by grants DMS-1502219, DMS-1760413, and
DMS-1800492, and DMS-2201045 of the United States National Science Foundation, and P.S. was partially supported by  
grants DMS-1665491 and DMS-1848562 of the United States National Science Foundation.

The authors thank Leo Gitin and Tom\'{a}s Ibarluc\'{\i}a for suggesting corrections to an earlier version of this paper.

\section{Some continuous logic} 
\label{sec:some-contunuous-logic}
In this section we recall some 
of the basics of continuous 
logic with a special emphasis on
the theory of interpretations.

\subsection{Continuous logic 
formalism}
We work in the setting of bounded continuous logic as presented in \cite{continuous-logic-book}.
Let us recall the basic notions.

\begin{defn}
A \emph{language} is a set \(\L\) of function symbols, constants and relation symbols
with the following additional data:
\begin{itemize}
    \item for each relation symbol $R$, an arity \(k_R\), a bounded interval
    $I_R \subseteq \mathbb R$ and a modulus of continuity $\Delta_R : (0,1]\to
    (0,1]$;
    \item for each function symbol $f$, an arity \(k_f\) and a modulus
    of continuity $\Delta_f : (0,1]\to (0,1]$;
    \item a real constant \(D_\L\). 
\end{itemize}

An \emph{\(\L\)-structure} is then a complete metric space $(M,d)$ of diameter at most
\(D_\L\) with, in addition,
\begin{itemize}
\item for each relation symbol $R$, a function  \(R^{M} : M^{k_R} \to I_R\)
which is \(\Delta_R\)-uniformly continuous;
\item for each \(k\)-ary function symbol \(f\), a function \(f^{M}: M^{k_f} \to M\)
which is \(\Delta_f\)-uniformly continuous;
\item for each constant symbol \(c\), an element \(c^{M}\) of \(M\).
\end{itemize}
\end{defn}

From now on, we assume that \(D_\L=1\) and that \(I_R = [0,1]\) for all \(R\). The
general case poses no extra difficulty.

\begin{defn}   
\label{def:formula}
\emph{Terms} of \(\L\) are built from the variables and constant symbols by composition with function symbols.

\emph{Formulas} of \(\L\) are formed from the relations and function symbols of
the language \(\L\), using terms, connectives -- that is, continuous functions \([0,1]^{\mathbb{N}}
\to[0,1]\) -- and quantifiers $\sup_x$ and $\inf_x$, for any variable $x$. We
allow formulas with countably many (free) variables.
\end{defn}

If \(x\) is a tuple of \(n\) 
variables, then we will use the 
notation \(M^x\) for the 
Cartesian power \(M^{|x|}\).
Note that we allow for 
the possibility that \(x\) is 
a countable tuple \(x = (x_0, x_1, x_2, \ldots )\).
We treat \(M^x\) itself as a 
complete metric space.
To be concrete, we define a 
metric by 
\( d_{M^x}(u, v) :=  \sup 2^{-i} d(u_i,v_i)\) where \( u_i \) is 
the \(i^\text{th}\) co\hyp{}ordinate
of \( u \).  When \(x\) is a 
finite tuple, this 
metric is equivalent to any of
the other natural choices, 
\emph{e.g.}, \(\vert\vert\cdot\vert\vert_\infty\) or \(\vert\vert\cdot\vert\vert_2\),
that one might prefer.
In the case where \(x\) is 
infinite, we need to be 
careful to ensure that the 
projection maps \(M^x \to M\) are uniformly continuous.
Given an \(\L\)-structure \(M\), a formula with free variables \(x\) is
interpreted in the natural way, yielding a uniformly continuous function
\(M^{x} \to [0,1]\). 

\begin{defn}
\label{def types}
Fix a language \(\L\). A \emph{partial type} in variables \(x\) is
a finitely consistent collection of conditions of the form \(\varphi(x) = 0\). A
\emph{type} is a maximal partial type.  A closed partial type, that is a partial type in the empty tuple of variables,
is called a \emph{theory}.
\end{defn}

\begin{defn}
\label{def:definable-predicate}
A \emph{definable predicate} on \(M^x\) is the interpretation of a formula \(\phi(x)\)
with free variables \(x\). 
 A subset \(X\subseteq M^x\) is \emph{definable} if \(d(x,X) = \inf_{y\in
X} d(x,y)\) is a definable predicate.  A \emph{definable function} is a function 
\(f:M^x \to M\) for which 
\( d(f(x),y) \) is a definable
predicate on \(M^{x,y}\).
\end{defn}

\begin{rem}
\label{rem:uniform-limits}    
In some presentations of continuous logic, a distinction is made between the definable 
predicates realized as the interpretations of formulas and uniform limits of 
such predicates.  Since we permit \( x \) to be a possibly countably infinite tuple of 
variables and allow all continuous functions on \([0,1]^x\) as connectives, there is no distinction
between these two classes of predicates.  Indeed, by restricting
to a subsequence, any uniform limit of predicates may be 
realized as a limit of predicates which are uniformly Cauchy for 
any given modulus of Cauchy uniformity.  Let us fix one such, 
\(\epsilon:\mathbb{N} \to (0,1] \).  The set 
\( C_\epsilon := \{ a \in [0,1]^{\mathbb{N}} : (\forall n, m) |a_n - a_m| \leq \epsilon(\min\{ n, m \}) \} \) is a closed subset of the 
normal space \([0,1]^{\mathbb{N}}\).  
The function \( f:C_\epsilon \to [0,1] \) given by \( f(a) := \lim_{n\to \infty}
a_n \) is continuous and bounded.  Thus, by the Tietze-Urysohn extension theorem
it extends to a continuous function 
\( F:[0,1]^{\mathbb{N}} \to [0,1] \).  Given an
\(\epsilon\)-Cauchy sequence of  
sequence \( (P_n(y))_{n=0}^\infty \) of definable predicates, 
we may compute $\lim_{n \to \infty} P_n(y)$ as the connective 
\( F \) applied to the sequence \( (P_n)_{n=0}^\infty \).
\end{rem}

\begin{rem}
\label{remark:first-order-as-continuous}
An \(\L\)-structure \(M\) in the 
sense of ordinary first\hyp{}order logic
may be regarded as a 
structure with respect to continuous 
logic by giving \(M\) the discrete metric
defined by \(d(x,y) = 0\) if \(x = y\)
and \(d(x,y) = 1\) if \(x \neq y\) and interpreting 
each relation symbol \(R\) in the variables
\(x\) as the uniformly continuous predicate \(R:M^x \to \mathbb{R}\) defined as \(R(a) = 0 \text{ if } M \models R(a) \) and 
\(R(a) = 1 \text{ if } M \models \neg R(a) \) where we may take \(1\) as the
modulus of uniform continuity.  
In this way, for any first\hyp{}order formula
\(\vartheta\) in the free variables 
\(x\) we have a predicate 
\(\vartheta:M^x \to \mathbb{R}\) also 
defined by \(\vartheta(a) = 0 
\text{ if } M \models \vartheta(a)\) 
and \(\vartheta(a) = 1\) otherwise. 
Because we may apply continuous 
connectives in order to produce new
definable predicates, there are many
other definable predicates even for 
this discrete structure.  A dense
set of such predicates may be 
obtained as follows.  Let 
\(\vartheta_1, \ldots, \vartheta_n\)
be a finite sequence formulas in the 
free variable \(x\) (which may be 
a tuple).  Let \(\alpha_1, \ldots, 
\alpha_n \in \mathbb{R}\) be 
a sequence of real numbers of the 
same length.  Then \( \sum_{i=1}^n 
\alpha_i \vartheta_i \) is a
definable predicate from 
\(M^x\) to \(\mathbb{R}\). 
\end{rem}

\subsection{Interpretations in 
continuous logic}
\label{sec:interp-continuous}

We need to extend the theory of interpretations to continuous logic.  For an
account in the context of (possibly infinitary) discrete logics
see~\cite[Chapter 5]{Hodges-mt}. We follow and expand on~\cite{BYKai}.

\begin{defn}
\label{def:fully-interpreted} 
Let \(M\) be an \(\L\)-structure. If \(\partial\) is a definable pseudometric
on some \(M^x\), we consider the quotient \(M^x/\partial\) with the associated
metric also denoted \(\partial\). A (quantifier\hyp{}free) definable predicate
on \(M^x/\partial\) is any real-valued function \((M^x/\partial)^y \to \mathbb{R}\) whose
pullback to \((M^x)^y\) is a (quantifier\hyp{}free) definable predicate.

As previously, \(X\subseteq (M^x/\partial)^y\) is definable if \(\partial(x,X)\)
is a definable predicate. Such an \(X\) is said to be interpretable. A
(quantifier\hyp{}free) definable predicate on \(X\) is the restriction of a
(quantifier\hyp{}free) definable predicate on \(M^x/\partial\).
\end{defn}

The full induced structure \(X^{\mathrm{full}}\) is obtained by naming every definable
predicate on~\(X\).

\begin{defn}
\label{def:interpretation-continuous}
An interpretation of the \(\L_2\)-structure \(N\) in the \(\L_1\)-structure
\(M\) is given by an \(\L_1\)-interpretable set \(\widetilde{N}\) and an
isometric $I:\widetilde{N} \to N$ so that for every \(\L_2\)-definable predicate
$Q:N^y \to \mathbb{R}$ the composite \(Q \circ I^y:\widetilde{N}^y \to
\mathbb{R} \) is a \(\L_1\)-definable predicate.  
\end{defn}

In other words, an interpreted structure is a reduct of \(X^{\mathrm{full}}\)
for some interpretable set \(X\).

\begin{rem}
\label{rem:alternative-interpretation}
As in the discrete case, it is sometimes useful to remember how the universe of
an interpretation appears as an interpretable set. Given an interpretation
\(I:\widetilde{N} \to N\) in the sense of
Definition~\ref{def:interpretation-continuous}, there is some definable
pseudometric \(\partial\) on \(M^x\) for some choice of variables \(x\) and a
definable predicate \(\varphi_{\partial_{\widetilde{N}}}\) on \(M^x\) inducing
the “\(\partial\)-distance to \(\widetilde{N}\)” predicate on \(M^x/\partial\).
If we let \(\pi_\partial:M^x \to M^x/\partial\) be the natural quotient map and
\(\widehat{N} := \pi_\partial^{-1} \widetilde{N} =
\varphi_{d_{\widetilde{N}}}^{-1} \{0\}\), then the data of \(I:\widetilde{N} \to
N\) is equivalent to that of the map \(I \circ \pi_\partial:\widehat{N} \to N\)
where now \(\widehat{I} := I \circ \pi_\partial\) is a surjective map which
pulls back the metric on \(N\) to the definable pseudometric on \(\widehat{N}\)
which is a definable set relative to that pseudometric.
\end{rem}

In Definition~\ref{def:interpretation-continuous} we wrote \(I^y\) for the
function on the Cartesian power \(\widetilde{N}^y\) of \(\widetilde{N}\) indexed
by the variables \(y\).  From now on we will simply write \(I\) for the natural
extension of \(I\) to Cartesian powers and restrictions to subsets.  Using these
natural extensions, we may compose interpretations.  If \(M_i\) is an \(\L_i\)
structure for \(i = 1\), \(2\), or \(3\) and we have interpretations
\(I:\widetilde{M_2} \to M_2\) and \(J:\widetilde{M_3} \to M_3 \) of \(M_2\) in
\(M_1\) and of \(M_3\) in \(M_2\) respectively, where \(\widetilde{M_2}
\subseteq M_1^x/\partial_2\) for some definable pseudometric \(\partial_2\) and
\( \widetilde{M_3} \subseteq M_2^y/\partial_3 \) for some definable pseudometric
\(\partial_3\), then \(J \circ I\), by which we mean \(J \circ (I^y/\partial_2)
\) where \(I^y/\partial_2\) is \(\pi_{\partial_2} \circ I^y\) with
\(\pi_{\partial_2}:M_2^y \to M_2^y/\partial_2\) being the natural quotient map,
is an interpretation of \(M_3\) in \(M_1\). Note that,
by~\cite[Proposition~3.6]{BY-reconstruction}, the pullback of the pseudometric
\(\partial_3\) on \(M_2\) to \(\widetilde{M_2}\) extends to a pseudometric on
\(M_1^x/\partial_2\) which we can pullback to \(M_1\) to define the double
interpretation.

\begin{rem}
\label{remark:hyperimaginaries-as-interpretable-structures}
We discussed in Remark~\ref{remark:first-order-as-continuous} how a first\hyp{}order 
structure may be regarded as a
structure for continuous logic. 
The structures interpretable in 
the usual sense of first\hyp{}order logic
are also interpretable in this 
extended sense.  However, there are 
other structures which may have 
non-discrete metrics which we can 
interpret in these discrete structures.
For example, the construction of 
hyperimaginaries in the sense of 
first\hyp{}order logic becomes a simple 
interpretation in continuous logic. 
\end{rem}

As with interpretations in the context of discrete structures, it is only necessary to check 
that the basic structure pulls back to definable 
predicates.

\begin{prop}
\label{prop:interpretation-basic-structure}
We are given two languages \(\L_1\) and \(\L_2\), an \(\L_1\)-structure \(M\),  
an \(\L_2\)-structure \(N\) with structural 
metric \(d_N\), an \(\L_1\)-interpretable set $\widetilde{N}$ in
\(M\), and an isometry (which for
us must be surjective) \(I:\widetilde{N} \to N\).  If \(I\) respects
the basic structure on \(N\) in the sense that each of 
\begin{itemize}
\item \( d_N(I(y),c) \) for \(c\) an \(\L_2\)-constant symbol,
\item \( d_N (I(y),f(I(z))) \) for \( f \) an \(\L_2\)-function symbol in the 
variables \( z \), and 
\item \( P(I(z)) \) for \( P \) an \(\L_2\)-predicate in the variable \( z\)
\end{itemize}
is a definable predicate on \(\widetilde{N}^z \), then \(I\) is an
interpretation of \(N\) in \(M\).  
\end{prop}

This proposition follows by induction on the construction of continuous 
logic formulas as in the classical case.  We leave the details 
to the reader.

As in the classical case, it makes sense to regard interpretations syntactically
so that one has a notion an interpretation of one theory in another.  That is,
if  we are given a formula \(\varphi_\partial(x,y)\) which is 
intended to define a pseudometric \(\partial\), 
another formula \(\varphi_{\widetilde{N}}(x)\) intended to define the 
\(\partial\)-distance to some definable set \(\widetilde{N}\), and 
formulas whose interpretations give each of the predicates
described in Proposition~\ref{prop:interpretation-basic-structure}
so that for each
\(\L_1\)-structure \(M\) which a model of some given theory \(T_1\) 
the formulas \(\varphi_\partial\) and \(\varphi_{\widetilde{N}}\) define 
a pseudometric and distance to definable set \(\widetilde{N}\) as intended
and we obtain
an \(\L_2 \)-structure \( N \) modeling a theory \( T_2 \) by taking the
universe of \( N \) to be the zero set \(\widetilde{N} :=  \{ a \in M^x/\partial :
\varphi_{d_{\widetilde{N}}}(a) = 0 \} \) with the metric given by \( \partial \)
 and the remaining structure described by the formulas corresponding to the
bullet points of Proposition~\ref{prop:interpretation-basic-structure}, then we
would say that we have an interpretation of \(T_2\) in \(T_1\).  We will abuse
notation writing \( I \) for the interpretation given by this choice of formulas
and then for any model \( M \) of \( T_1 \) we will write \( I(M) \) for the
model of \( T_2 \) given by these formulas.



The syntactic presentation of an interpretation gives rise 
through the proof of 
Proposition~\ref{prop:interpretation-basic-structure} 
to a translation \( P \mapsto P^I \) from the interpreted
language \(\L_2\) to the language \( \L_1 \) having the
property that for any \(\L_2\)-predicate \(P\) we have 
\(P^I = P \circ I \).

On the face of it, two interpretations given by equivalent (modulo \(T_1\))
formulas are the same in that the predicates are the same.  However, there may
be cases in which it pays to be more careful about the choice of the formulas.  
For example, if we can take the formulas to be simple enough, 
then an interpretation would restrict to give interpretations
for substructures.  

\begin{prop}
\label{prop:qf-interp}
With the notation as in 
Proposition~\ref{prop:interpretation-basic-structure}, 
if 
\begin{itemize} 
\item the predicates giving
the \(\partial\)-distance to \(\widetilde{N}\), the pullback of the distance \(
d_N \), and the pullback of the distinguished predicates on \(N\) are
quantifier\hyp{}free definable,
\item for each constant symbol \(c\) in \( \L_2\) there is a uniform limit of closed
\(\L_1\)-terms \( \widetilde{c} \) so that \(I(\widetilde{c}^M) = c^N \), and 
\item for each function symbol \(f\) with variable \(y\) in 
\( \L_2 \) there is a uniform limit of \(\L_1\)-terms \(\widetilde{f}\) so 
that \(I \circ \widetilde{f} = f \circ I\),
\end{itemize}
then the restriction of 
$I$ to any substructure of \( M \) in an interpretation of a 
substructure of \( N \).
\end{prop}

We will call the interpretations of Proposition~\ref{prop:qf-interp}
\emph{quantifier\hyp{}free interpretations}.  Note that 
requirement that the basic function symbols and constant symbols
be represented by terms is stronger than merely asking that
the pullback of the distance to the graphs 
of the interpretations of
the function symbols and to the interpretations of the 
constant symbols be quantifier\hyp{}free definable predicates.
The translations associated to 
quantifier\hyp{}free interpretations preserve the quantifier
complexity of predicates.

\begin{defn}
As with the classical theory, we say that two interpretations
\(I\) and \(J\) of \(N\) in \(M\) are \emph{homotopic} if the function \(J^{-1} \circ I \)  
from the domain \( \widetilde{N}_I \) of \(I\) and the domain \( \widetilde{N}_J
\) of \(J\) is definable. In the special case that we have a pair of
interpretations \(I\) of \(N\) in \(M\) and \(J\) of \(M\) in \(N\), we say that
this is a \emph{bi\hyp{}interpretation} if the composite \(I \circ J\) is
homotopic to the identity interpretation of \(M\) in itself and \(J \circ I\) is
homotopic to the identity interpretation of \(N\) in itself.
\end{defn}

\subsubsection{Preservation of 
definability by interpretations}
\label{sec:preservation-definability}
An important consequence of a pair of interpretations
forming a bi\hyp{}interpretation is that the image of a 
definable set under a definable predicate is itself
definable in the interpreted model.  In fact, we only need 
half of the condition on being a bi\hyp{}interpretation for 
this fact to hold. 

\begin{prop}
\label{prop:bi-interp-push}
Suppose that \(I\) is an interpretation of \(N\) in \(M\) and 
that \(J\) is an interpretation of \(M\) in \(N\).  Suppose 
moreover that \(J \circ I\) is homotopic to the identity 
interpretation of \(M\) in \(M\).  If \(P\) is an \(\L_2\)-predicate, then the function \(\overline{P}\)
on \(M\) defined by \(\overline{P}(x) := P(J^{-1}(x))\) is 
an \(\L_1\)-definable predicate. 

\end{prop}
\begin{proof}
Let \(\alpha:(J \circ I)(M) \to M  \) be the definable 
isomorphism witnessing that 
\( J \circ I \) is homotopic
to the identity self-interpretation of \( M \).
Since \(I\) is an
interpretation, \(P^I  :=  P \circ I\) is an \( \L_1\)-definable 
predicate.  Unwinding the definition of \( \overline{P} \) we 
see that \(\overline{P} = P^I \circ \alpha \) is 
also an \( \L_1 \)-definable predicate.
\end{proof}

As a corollary of
Proposition~\ref{prop:bi-interp-push}
we have the following consequence 
on the images of definable sets.

\begin{cor}
\label{cor:image-definable}
Under the hypotheses of Proposition~\ref{prop:bi-interp-push}, if \(X \subseteq
\widetilde{M}^y\) is an \(\L_2\)-definable set, 
then \(J(X) \subseteq M^y\) is an
\(\L_1\)-definable set.
\end{cor}
\begin{proof}
Let \(P(y) = d_N(y,X) \) be the \(\L_2\)-definable predicate expressing the
distance to \( X \).  By
Proposition~\ref{prop:bi-interp-push}, \(\overline{P} = P \circ J^{-1} \) is \(
\L_1\)-definable.  
Since \( J \) is an isometry, we have \(d_M(z,J(X)) = d_N(J^{-1}(z), X) =
P(J^{-1}(z)) = \overline{P}(z) \). That is, the distance to \(J(X)\) is an
\(\L_1\)-definable predicate, showing that \(J(X)\) is \(\L_1\)-definable.
\end{proof}


\subsubsection{Preservation 
of existential closedness 
by interpretations}
\label{sec:preserve-ec}

A quantifier\hyp{}free 
interpretation of one theory in 
another will transform existentially 
closed extensions to existentially closed
extensions.  Before we prove this 
proposition, let us recall what we mean 
by an existentially closed extension in 
continuous logic.

\begin{defn}
\label{def:ex-ext}
The \(\L\)-structure \(M_1\) is 
existentially closed in 
\(M_2\), written \(M_1 \preceq_\exists M_2 \), if \(M_1\) is a substructure of 
\(M_2\) and for every pair of 
variables \( x\) and \(y\), 
quantifier\hyp{}free \( \L \) predicate
\(P(x,y)\) in the variables \((x,y)\) 
taking values in \([0,\infty)\), and 
point \(a \in M_1^y\), if \(\inf_x P^{M_2}(x,a) = 0 \), then  \( \inf_x P^{M_1}(x,a) = 0 \).  

We say that \( M_1 \) is existentially 
closed relative to the theory \( T \) if
for every extension \( M_2 \models T\) of \(M_1\), we have 
\(M_1 \preceq_\exists M_2\).
\end{defn}

\begin{prop}
\label{prop:ec-to-ec}
Let \( I \) be a 
quantifier\hyp{}free interpretation of the 
\( \L_2 \) theory \( T_2 \) in the 
\( \L_1 \) theory \( T_1 \).  If 
\(M_1 \preceq_\exists M_2 \) is 
an existentially closed
extension of models of \( T_1 \),
then 
\( I(M_1) \preceq_\exists I(M_2) \).
\end{prop}

\begin{proof}
Let us write \( \widetilde{N} \) for the definable set 
in the variables \( x \) giving the universe of the interpretation \( I \).  
Let \(Q(y,z)\) be a quantifier\hyp{}free
\(\L_2\)-predicate and \( a \in I(M_1)^z \)
a point so that \( \inf_y Q^{I(M_2)}(y,a) = 0\). 
Let \( \widetilde{a} \in \widetilde{N}(M_1) \) with \(I(\widetilde{a}) = a \).  
Then \( \inf_z (Q^I)^{M_2}(z, \widetilde{a}) + \partial^{M_2}_{\widetilde{N}}(z)
= 0 \) and \( Q^I + \partial_{\widetilde{N}}\) pulls back to a
quantifier\hyp{}free predicate as \(I\) is quantifier\hyp{}free.
Since $M_1 \preceq_\exists M_2$, we have \( \inf_{z} (Q^I)^{M_1}(z,
\widetilde{a}) + \partial^{M_1}_{\widetilde{N}}(z) = 0\) as well. As
\((Q^I)^{M_1}\) is uniformly continuous, it follows that \( \inf_{z\in
\widetilde{N}^y} (Q^I)^{M_1}(z, \widetilde{a}) = 0\). Unwinding the meaning of
this equality, we have \( \inf_{I(M_1)^y} Q^{I(M_1)}(y,a) = 0\), as required.
\end{proof}

It follows from Proposition~\ref{prop:ec-to-ec} that a quantifier\hyp{}free
bi\hyp{}interpretation between theories will take existentially
closed models to existentially closed models.  Let us express this 
result using a weaker condition than bi\hyp{}interpretation.

\begin{prop}
\label{prop:interp-full-ec}
Let \( I \) be a 
quantifier\hyp{}free interpretation of the 
\( \L_2 \) theory \( T_2 \) in the 
\( \L_1 \) theory \( T_1 \) having the property that for every 
extension \( N_1 \subseteq N_2 \) of models of \( T_2 \) there is
an extension \( M_1 \subseteq M_2 \) of models of \( T_1 \) and isomorphisms \( \rho_i: I(M_i) \cong N_i \)
for \(i = 1 \) and \( 2 \) fitting into the following commuting square.
\[ \begin{tikzcd} I(M_2) \ar[rr,"{\rho_2}"] && N_2 \\
I(M_1) \ar[u,hook] \ar[rr,"{\rho_1}"] && N_1 \ar[u,hook] 
\end{tikzcd}
\]
If \( M \) is an existentially closed model with respect to \( T_1 \), then 
\( I(M) \) is existentially closed with respect to \( T_2 \).
\end{prop}
\begin{proof}
Let \( M \) be an existentially closed model with respect to \( T_1 \). Since \( I \) interprets \( T_2 \) in \( T_1 \),  \( I(M) \models T_2 \).   Let 
\(I(M) \subseteq N' \) be an extension of models of \( T_ 2 \). By our hypothesis, possibly after applying 
an isomorphism, we may find an extension \(M \subseteq M' \) of models of \(T_1\) so that the extension 
\(I(M) \subseteq N' \) is the extension \(I(M) \subseteq I(M')\) where the inclusion is the extension 
is obtained from the extension \(M \subseteq M' \) via the functoriality of \( I \).  Since \( M \) is 
existentially closed in the class of models of \( T_1 \), \(M  \preceq_\exists M' \).  By 
Proposition~\ref{prop:ec-to-ec}, \(I(M) \preceq_\exists I(M') = N' \).  Hence, \( I(M) \) is 
existentially closed with respect to \( T_2 \)
\end{proof}

\begin{cor}
\label{cor:binterp-ec} If \( I \) is  half of a quantifier\hyp{}free bi\hyp{}interpretation of the 
the theory \( T_2 \) in the theory \( T_1 \) and \( M \) is existentially closed with 
respect to \( T_1 \), then \( I(M) \) is existentially closed with respect to \( T_2 \).
\end{cor}

\begin{proof}
Let \( J \) be the other half of the bi\hyp{}interpretation.  Then any extension
\(N_1 \subseteq N_2 \) of models of \( T_2 \) is isomorphic to \( I (J(N_1))
\subseteq I (J(N_2) ) \).  Hence, Proposition~\ref{prop:interp-full-ec} applies.
\end{proof}

Sometimes we may check existential 
closedness relative to a weaker 
theory by considering only 
existential closedness relative to 
a stronger one. The following lemma
will be used in Section~\ref{sec:FW}.

\begin{lem}
\label{lem:cotheory-ec}    
If \(T \subseteq T'\) are 
\(\L\) theories having the property
that for any model \(M \models T\) of 
\(T\)
there is an extension \(M \subseteq N \models T'\) to a model of \(T'\), 
then a model \( M \models T'\) is 
existentially closed with respect to 
\(T\) if and only if it is existentially
closed relative to \(T'\).
\end{lem}
\begin{proof}
This is proven just 
as it would be in the first\hyp{}order 
case.
\end{proof}

\subsection{Types in continuous logic}
\label{sec:types-continuous}

We modify the notion of type 
spaces from~\cite[Section 3]{continuous-logic-book} slightly. Recall \cref{def:fully-interpreted} of an interpretable set and its induced predicates.

Let \(M\) be an \(\L\)-structure and \(X\) be an interpretable set in \(M\).
Recall (\cref{def types}) that a type in \(X\) is a maximal consistent
collection of conditions “\(\phi(x) = 0\)” where \(x\) ranges over \(X\).

\begin{defn}
\label{def:type-space}
We write \(S_X\) for the set of types in \(X\). We give \(S_X\) the weakest
topology so that for each \( \varphi \) the function \( S_X \to \mathbb{R} \)
defined by \(p \mapsto \varphi^p := \inf_r \phi(x) \dotminus r \in p \) is
continuous. If \(A \subseteq M\) is a subset of \(M\), then \(S_X(A)\) is the
type space \(S_X\) relative to the \(\L_A\)-structure \(M_A\) where the new
constant symbols \( a \in A\) are interpreted by \(a^{M_A} := a\).
\end{defn}

Since each formula \(\varphi(x)\) is constrained to take values in a compact
interval \(I_\varphi\), evaluation at all of the formulas in the variable \(x\)
realizes \(S_X\) as a closed --- and hence, compact --- subset of the product
\(\prod_\varphi I_\varphi\) of compact intervals over the set of all such
formulas. Thus, \(S_X\) is itself a compact Hausdorff space.

Since, as discussed in 
Remark~\ref{rem:uniform-limits},
we allow for a process of taking 
uniform limits as a connective, 
using~\cite[Proposition 3.4]{continuous-logic-book}, we see that the space
\(C(S_X,\mathbb{R})\) of continuous 
real-valued functions on the type space
\(S_X\) may be identified with the set of 
\(\L\) formulas on \(X\) via evaluation.
This observation yields the following 
proposition about maps on type spaces
induced by interpretations.

\begin{prop}
\label{prop:interp-type}
If \(I:\widetilde{N} \to N\) is
an interpretation of the \(\L_2\)-structure
\(N\) in the \(\L_1\)-structure, then 
\(I\) induces a continuous map 
\(I_*:S_{\widetilde{N}} \to S_N\).
\end{prop}
\begin{proof}
For \(p \in S_{\widetilde{N}}\), we 
define \(I_*(p)\) by specifying 
\(\varphi^{I_*(p)} := (\varphi^I)^p\) 
for each \( \varphi \) an \(\L_2\)-formula
in the variable \(y\) ranging over \(N\).  
Continuity of \(I_*\) is an immediate consequence of the fact that
precomposition of \(I\) with a definable predicate is itself a definable
predicate:  the basic open subsets of \(S_N\) take the form \([r < \varphi < s]
:= \{ p \in S_N : r < \varphi^p < s \} \) for \( \varphi \) an \(\L_2\)-formula
in the variable \(y\) and \(r < s\) real numbers.  
The preimage of this set under \(I_*\) is \([r < \varphi^I < s ]\)  which is open in 
\(S_{\widetilde{N}}\).
\end{proof}

When 
\(I\) is half of 
a bi\hyp{}interpretation 
we may upgrade
Proposition~\ref{prop:interp-type} to the 
assertion that \(I_*\) is a 
homeomorphism. 
Such an assertion 
would fail if
we were to treat 
interpretations 
as in Remark~\ref{rem:alternative-interpretation}.

\begin{prop}
\label{prop:homeo-type-spaces}
Let \(I:\widetilde{N} \to N\) and 
\(J:\widetilde{M} \to M\) be a pair of 
interpretations forming a bi\hyp{}interpretation.
Then \(I_*:S_{\widetilde{N}} \to S_N\)
is a homeomorphism.
\end{prop}
\begin{proof}
We already
know that \(I_*\) is continuous.  Let us check that
it is open.  Let \(P:\widetilde{N} \to \mathbb{R}\) be a
definable predicate and let \(r < s\) be real numbers.  
By Proposition~\ref{prop:bi-interp-push}, 
\(P^{I^{-1}} := P \circ I^{-1} : N \to \mathbb{R} \) is
a definable predicate.  
From the definition of \(I_*\), we see that
\(I_* ([r < P < s]) = [r < P^{I} < s ]\), implying
that \(I_*\) is an open mapping, and, hence, 
a homeomorphism.
\end{proof}

Under an interpretation, the pullback of a definable set is a definable set.
It follows that in Proposition~\ref{prop:homeo-type-spaces} for any definable
set \(X \subseteq N^y \), there is a definable set \(\widetilde{X} := I^{-1} X\)
for which \(I\) induces a homeomorphism \(S_{\widetilde{X}} \to S_X\).  We may
go further.  
If \(\partial:X \times X \to \mathbb{R}_{\geq 0}\) is an \(N\) definable 
pseudometric on \(X\), then \(\partial^{I}\) is an \(M\) definable
pseudometric on \(\widetilde{X}\) and \(I\) induces a map between the 
interpretable structures \(\widetilde{X}/\partial^I\) and \(X/\partial\) 
obtained by quotienting by the induced equivalence relations.  
The induced map on the type spaces \(S_{\widetilde{X}/\partial^I} \to S_{X/\partial}\) is a homeomorphism.

\section{Metric valued fields}
\label{sec:MVF}

\begin{defn}
Let \(\Lrg\) denote the bounded continuous ring language with binary function symbols \(+\) and \(\cdot\), unary function symbol \(-\) and constants \(0,1\).
\end{defn}

The distance takes values in \([0,1]\) and all three symbols come with the modulus of continuity \(x\mapsto 2x\).

\begin{defn}
Let \((K,v)\) be a valued field which admits a rank one coarsening \(v_0 : K\to (\Rr,+)\). We consider the \(\Lrg\)-structure \((\Val(K),|x-y|,+,\cdot,-,0,1)\) where \(|x| = e^{-v_0(x)}\).
\end{defn}

\begin{lem}
The class \(\MVF\) of all such structures is elementary. It is axiomatized by:
\begin{itemize}
\item the (universal) theory of ultrametric domains;
\item the axiom \(\sup_{x,y}\inf_z \min(xz-y,yz-x) = 0\).
\end{itemize}
\end{lem}

\begin{proof}
Any structure as above is a model of that theory. Conversely, any model \(R\) of
that theory is a valuation ring and the open unit ball for \(\vert\cdot\vert\) is an ideal,
\emph{i.e.}\ \(\vert\cdot\vert\) is a (necessarily rank one) coarsening of the valuation
associated to \(R\).
\end{proof}

\begin{rem}
Let \((K,v)\) be a valued field. For every \(\gamma\in vK\), let \(\gamma\Val\)
denote \(\{x\mid v(x)\geq \gamma\}\). Fix some \(\gamma\in vK^\times\) and
consider the ring \(\Val_\gamma := \projlim_{n} \Val/\gamma^n \Val\) which is
naturally isomorphic to the residue field for the coarsening of \(v\) by the
smallest convex subgroup containing \(\gamma\).

Then, after an choice of normalization, \(\Val_\gamma\) is a model of \(\MVF\) and every model of
\(\MVF\) is obtained this way. So \(\MVF\) is the theory of a continuously
interpretable structure in the discrete theory of valued fields.
\end{rem}

\begin{defn}
Let \(\ACMVF\) be the \(\Lrg\)-theory of algebraically non-discrete metric valued fields which consists of \(\MVF\) and the axioms:
\begin{itemize}
    \item \(\inf_x \max\{|x| \dotminus 1/2,1/2 \dotminus |x|\} = 0\),
    \item \(\sup_x\inf_y |P(x,y)| = 0\), for any finite tuple of variables \(x\), any variable \(y\) and any polynomial \(P\in\mathbb{Z}[x,y]\) monic (non-constant) in the variable \(y\).
\end{itemize}

Fix \(p\) a prime.  We write
\( \ACMVF_{0,p} \) for the 
\(\Lrg\) theory of algebraically closed non-discrete metric valued fields of characteristic zero with residue field of characteristic zero and 
\( \ACMVF_p\) for the 
\(\Lrg\) theory of algebraically closed non-discrete metric valued fields of characteristic $p$. For fixed  \(\alpha \in (0,1)\) we let
\(\PERF_{|p|= \alpha}\) be the theory of (valuation rings of) mixed characteristic
\((0,p)\) perfectoid fields which consists of \(\MVF\) and the axioms:
\begin{itemize}
    \item \(|p| = \alpha\),
    \item \(\inf_x \max\{|x|\dotminus \alpha^{1/p},\alpha^{1/p} \dotminus |x|\}=0\),
    \item \(\sup_x\inf_y\inf_z |x - y^p - pz| = 0\).
\end{itemize}
Note that our axioms for mixed
characteristic perfectoid fields depend on a choice of normalization.

The theory \(\PERF_{|p|=0}\) is the theory of (valuation rings of)
characteristic \(p\) perfectoid fields  which consists of \(\MVF\) and the
axioms:
\begin{itemize}
    \item \(|p| = 0\),
    \item \(\inf_x \max\{|x|\dotminus 1/2,1/2 \dotminus |x|\}=0\),
    \item \(\sup_x\inf_y |x - y^p| = 0\).
\end{itemize}
Note that all models of \(\PERF_{|p|=0}\) are perfect.
\end{defn}

Let \(\rD(x,y)\) be the predicate \(\inf_{z\in \cO} |y-xz|\). Note that for every \(x,y\in\cO\),
\[|x - y\cO^\times| =
\left\{\begin{array}{ll}0& \text{if } v(x) = v(y)\\
\max(|x|,|y|)& \text{otherwise}\end{array}\right.
= \max(D(x,y),D(y,x)).\]

This is a pseudo metric whose associated imaginary is \(v\cO\) --- with some non
discrete norm. The induced topology is discrete at every point except for
\(v(0)\). A basis of neighborhoods of \(v(0)\) is given by upward closed sets.

Also, this shows that \(\cO^\times\) is definable, and since \(|x-\fm| =
1\dotminus d(x,\cO^\times)\), so is \(\fm\). The pseudo norm \(|x-y-\fm|\) is
thus also a predicate and the associated imaginary is the residue field \(Kv\)
with the discrete norm. Similarly, for any non-zero \(\varpi\in\Val\), the set
\(\varpi\Val\) is definable and the pseudo-norm \(|x-y-\varpi \Val|\) gives rise
to the imaginary \(\Val/(\varpi)\) with the norm induced by \(\vert\cdot\vert\) --- the
associated topology is trivial.

\begin{defn}
Let \(\LD\) denote \(\Lrg\cup\{\rD\}\).
\end{defn}

\begin{prop}\label{ACMVF EQ}
The theory \(\ACMVF\) eliminates quantifiers in \(\LD\) and is the model completion of \(\MVF\).
\end{prop}

Let \(\Ldiv\) denote the
discrete language of rings with a divisibility predicate interpreted in valued
fields as \(v(x)\leq v(y)\). Recall the (discrete) \(\Ldiv\)-theory \(\ACVF\) of non-trivially valued algebraically closed fields eliminates quantifiers.

\begin{proof}
First note that in any \(\aleph_0\)-saturated \(K \models\ACMVF\), we have
\(|\cO| = [0,1]\). Indeed, if \(a\in\cO\) is such that \(|a|\in (0,1)\), then
\(\Qq\cdot |a|\subseteq |\cO|\) is dense in \([0,1]\).

\begin{claim}
Let \(M\models\ACMVF\) be \(\kappa\)-saturated (with induced valuation \(v\)).
There exists a \(\kappa\)-saturated \(N\models\ACVF\) containing \((M,v)\) such
that \(\cO(M)\) is a section of the map \(\cO(N) \to \cO(N)/ I_M\) where
\(I_M:=\{x\in \cO(N)\mid v(x) > v(M)\}\). 
\end{claim}

\begin{proof}
Let \(K\) be the fraction field of \(\cO(M)\), let \(\Gamma\) be a \(\kappa\)-saturated divisible ordered abelian group and let \(N := K((\Gamma))\) with the valuation \(w(\sum_{i\geq i_0} a_i t^i) = (v(a_{i_0}),i_0) \in v(K)\times \Gamma\) if \(a_{i_0}\neq 0\). One can check that \(w(N)\) is a \(\kappa\)-saturated divisible ordered abelian group. The residue field of \((N,w)\) is isomorphic to the residue field of \((K,v)\) which is \(\kappa\)-saturated. Since \(N\) is maximally complete, \(N\) is \(\kappa\)-saturated.
\end{proof}

Let now \(M,N\models\ACMVF\), let \(A\leq M\) and \(f : A\to N\) be an \(\LD\)-embedding.
Assume that \(N\) is \(|M|^+\)-saturated. Let \(N_1\supseteq N\) be as in the
claim. If \(\vert\cdot\vert\) is discrete on \(A\), let \(c\in M\) be such that \(|c| \in
(0,1)\) and \(d\in N\) such that \(|d| = |c|\) --- note that \(f\) then extends
to an \(\Ldiv\)-embedding sending \(v(c)\) to \(v(d)\). Since \(D(x,y) = 0\)
defines the usual divisibility predicate of \(\ACVF\), by quantifier
elimination, \(f\) extends to an \(\Ldiv\)-embedding \(g : M \to N_1\) ---
sending \(v(c)\) to \(v(d)\) if \(A\) is discrete. Since \(g(M)\cap I_N =
g(M\cap I_M) = \{0\}\), this embedding induces an \(\LD\)-embedding \(h : M
\to N\) extending \(f\). Quantifier elimination (and the rest of the statement)
follows.
\end{proof}

\begin{rem}
A similar reduction shows quantifier elimination for \(\ACMVF\) in the three
sorted language where we add the value group and residue field.
\end{rem}

Let us conclude this section by noting the relationship between the discrete
structure and the continuous structure of a perfectoid field. This relies on a
deep result of Jahnke and Kartas \cite{JahKar-Perf}.

Fix \(\alpha < 1\) and let \(M\models \PERF_{|p|=\alpha}\) with associated
valuation \(v\). Let us assume that \(v\) is henselian and \(v(M)\) has bounded
regular rank. An ordered group has bounded regular rank if all definable convex
subgroups (in elementary extensions) are \(\emptyset\)-definable --- these
groups are also known as groups with finite spines. Note that if \(v\) is a rank
one valuation, \emph{i.e.} coincides with \(-\log \vert\cdot\vert\), then both hypotheses are
verified.

\begin{lem}
\label{bd rk div}
Let \(\Gamma \preccurlyeq \Gamma'\) be an elementary extension of bounded
regular rank groups (in the discrete language of ordered groups). Let \(\Delta\)
be the convex hull of \(\Gamma\) in \(\Gamma'\). Then \(\Gamma'/\Delta\) is
divisible.
\end{lem}

\begin{proof}
Otherwise, if \(\gamma\nin \Delta +
n\Gamma\), then the largest convex subgroup \(H\) such that \(\gamma\nin H +
n\Gamma\) is definable in \(\Gamma'\) (see \cite[Lemma~2.1]{CluHal-OAG}) but it
cannot be \(\emptyset\)-definable as \(\Gamma \leq \Delta\leq H < \Gamma'\).
\end{proof}

Let \(\fU\) be a non-principal ultrafilter. Let \(M^\fU\) denote the discrete
ultrapower of \(\Ldiv\)-structure and let \(M^\fU_0\) denote the continuous
ultrapower of \(\Lrg\)-structures. If \(\varpi\in \Val(M)\setminus\{0\}\) is
topologically nilpotent, then \(M^\fU_0\) is isomorphic to the residue field of
\(\Val(M^\fU)[1/\varpi]\), whose valuation we denote \(v_0\). Also,
\(v(M^\fU_0)\) is the convex hull of \(v(M)\) inside \(v(M^\fU)\).

By \cite[Theorem~4.3.1]{JahKar-Perf}, \((M^\fU,v_0)\) is henselian with perfect
residue field. So we can find a lift \(f : M^\fU_0\to M^\fU\). It is also an
embedding of valued fields.

\begin{prop}
\label{Ldiv vs Lrg}
The \(\Ldiv\)-embedding \(f : (M^\fU_0,v)\to (M^\fU,v)\) is elementary.
\end{prop}

\begin{proof}
This follows immediately from \cite[Theorem~5.1.4]{JahKar-Perf}. Note that,
looking at the proof of \cite[Theorem~5.1.2]{JahKar-Perf}, the assumption that
\(v(M^\fU_0) \preccurlyeq v(M^\fU)\) can be lifted given \cref{bd rk div}.
\end{proof}

\begin{cor}
\label{Ldiv equiv Lrg}
Let \(M,N\models \PERF_{|p|=\alpha}\) whose valuation is henselian and whose value group has bounded regular rank. 
\begin{enumerate}
\item If \(M\) and \(N\) are \(\Lrg\)-elementarily equivalent then they are
\(\Ldiv\)-elementarily equivalent.
\item Let \(\varpi_M\in \Val(M)\) (resp. \(\varpi_N\in \Val(N)\)) be such that
\(0 < |\varpi_M| = |\varpi_N| < 1\). If \((M,\varpi_M)\) and \((N,\varpi_N)\)
are \(\Ldiv\)-elementarily equivalent, they are also \(\Lrg\)-elementarily
equivalent.
\end{enumerate}
\end{cor}

\begin{proof}
Let us first assume that \(M\) and \(N\) are equivalent as \(\Lrg\)-structure.
Then for some non-principal ultrafilter \(\fU\), by the Keisler-Shelah theorem,
the \(\Lrg\)-ultrapowers  \(M^\fU_0\) and \(N^\fU_0\) are isomorphic. By
\cref{Ldiv vs Lrg}, the \(\Ldiv\)-ultrapowers  \(M^\fU\) and \(N^\fU\) are
elementarily equivalent, and hence so are the \(\Ldiv\)-structures \(M\) and
\(N\).

The second item follows from the fact that, as noticed above, the
\(\Lrg\)-structure is continuously interpretable in the \(\Ldiv\)-structure
(with parameter \(\varpi\)).
\end{proof}

\begin{rem}
\begin{enumerate}
\item Under the hypotheses of \cref{bd rk div}, we can prove that
\(\Delta\leq\Gamma'\) is elementary. \Cref{Ldiv equiv Lrg} would hold for any
class of value groups with this property. However, we do not know if this holds
beyond bounded regular rank.
\item It is not clear either if assuming that the valuation is henselian is
necessary in \Cref{Ldiv equiv Lrg}.
\end{enumerate}
\end{rem}

\section{Bi-interpretations}
\label{sec:bi-interpretations}


If \(R\) is a ring, let \(R^\flat = \projlim_{x\mapsto x^p} R\). A priori it is
only a multiplicative monoid, unless \(R\) has characteristic \(p\), in which
case it is a ring.

\begin{lem}
\label{crucial}
Let \(M\models \PERF_{|p| = \alpha}\) for some \(\alpha < 1\). Let \(\varpi\in
\Val(M)\setminus\{0\}\) be topologically nilpotent divisor of \(p\). The natural
map \(f: \Val(M)^\flat \to (\Val(M)/(\varpi))^\flat\) is a bijection. The
inverse map is given by \(g : x\mapsto (\lim_{j} \tilde{x}_{i+j}^{p^j})_i\),
where \(\tilde{x}_{i}\) reduces to \(x_i\) mod \(\varpi\).
\end{lem}

In particular, \(\Val(M)^\flat\) can be made into a ring.

\begin{proof}
The crucial fact that is the following: if $a,b\in \Val(M)$ are congruent modulo
$\varpi^i$, then $a^p\equiv b^p \mod \varpi^{i+1}$. It follows that if $a\equiv
b \mod \varpi$, then $a^{p^i} \equiv b^{p^i} \mod \varpi^{i+1}$.

Now consider some \(x = (x_i)_i \in (\Val(M)/(\varpi))^\flat\), since
\(x_{i+j}^{p^j} = x_i\), the sequence \(\tilde{x}_{i+j}^{p^j}\) is Cauchy and
its limit $y_i\in\Val(M)$ does not depend on the choice of $\tilde{x}_{i}$. Note
also that \((\lim_{j} \tilde{x}_{i+j+1}^{p^j})^p = \lim_{j}
\tilde{x}_{i+j+1}^{p^{j+1}}= \lim_{j} \tilde{x}_{i+j}^{p^j}\) so
\((y_i)_{i\in{\mathbb{N}}}\) is indeed an elements of \(\cO(M)^\flat\). Since \(y_i
\equiv x_i \mod \varpi\), we have \(f\circ g = \id\). Finally, for every \(x \in
\cO(M)^\flat\), we have \(\lim_j x_{i+j}^{p^j} = x_i\) and hence \(g\circ f =
\id\).
\end{proof}

Let \(M \models \PERF_{|p|=\alpha}\), where \(0 < \alpha < 1\). The projection
\(\res_p : \Val(M) \to \Val(M)/(p)\) is an interpretation of the discrete ring
\(\Val(M)/(p)\) in \(M\). Our goal now is to axiomatize \(\Val(M)/(p)\). Let
\(\TVR\) (truncated valuation ring) be the first other theory in the language
\(\Ldiv\) of (discrete) rings with a divisibility predicate \(x \div y\)
axiomatized by:
\begin{itemize}
\item The theory of rings;
\item for all \(x,y\), we have \(x\div y\) if and only if \(x\) divides \(y\);
\item for all \(x\) and \(y\), either \(x\div y\) or \(y\div x\);
\item for all \(x\) and \(y\), if \(x\neq 0\) and \(x=yx\) then \(y\) is a unit.
\end{itemize}

Let \(R\models \TVR\), we can consider the map \(v : a \mapsto (a)\) into the
the set \(\Gamma\) of principal ideals of \(R\) ordered by reverse inclusion ---
so \(v(x)\leq v(y)\) if and only if \(x\div y\). It is a linearly ordered
commutative monoid with respect to multiplication of principal ideals (that we
denote +). Its minimal element is the neutral element \(0 = v(1)\). The maximal
element is \(\infty = v(0)\), we have \(\gamma + \infty = \infty\) for all
\(\gamma \in \Gamma\) and for every \(x\in R\), if \(v(x) = \infty\) then \(x =
0\). Moreover, for every \(\gamma,\delta\in R\), if \(\gamma \leq \delta \neq
\infty\) then there exists a unique \(\epsilon \in \Gamma\) such that
\(\gamma + \epsilon  = \delta\) --- we also write \(\epsilon = \delta-\gamma\).
Indeed, if \(\delta = (y)\subseteq (x)=\gamma\), there exists \(z\in R\) such
that \(y = xz\); and if, for some \(z_1,z_2\in R\), we have \(xz_1 = xz_2\neq
0\), then we may assume \(z_2 = z_1 a\) for some \(a\) and thus \(xz_1 = xz_1
a\). So \(a\in R^{\times}\) and \(v(z_1) = v(z_2)\). Finally, for every \(x,y\in
R\), we have \(v(xy) = v(x)+v(y)\) and \(v(x + y) \geq \max\{v(x),v(y)\}\). We
say that \(v\) is a truncated valuation on \(R\).

Fix some \(\alpha \in (0,1)\). If \(\varpi\in R\setminus\{0\}\) be topologically nilpotent, then \(v\) induces a unique ultrametric
norm \(\vert\cdot\vert : \Gamma \to (0,1)\) such that \(|\varpi| = \alpha\).

\begin{lem}
Then norm \(\vert\cdot\vert\) is continuously quantifier\hyp{}free \(\varpi\)-definable in \(R\).
\end{lem}

\begin{proof}
We present the proof under the further 
assumption 
that \(p^{-n}v(\varpi)\) exists for all \(n\geq 0\), as this will
simplify the notation and is the case we will use afterwards. Fix some
\(n\geq 1\). For every \(x\in R\) and \(1\leq m < p^{2n}\), if \(m/(p)^n \cdot
v(\varpi) \leq v(x) < (m+1)/(p)^n \cdot v(\varpi)\), then \(\alpha^{(m+1)/(p)^n}
\leq |x|\leq \alpha^{m/(p)^n}\). In that case, set \(t_n(x) = \alpha^{m/(p)^n}\)
--- if \(v(x) \geq p^n\cdot v(\varpi)\), set \(t_n(x) = 0\). Then the \(t_n\)
are quantifier\hyp{}free \(\varpi\)-definable as they take finitely many values and
the fibers are quantifier\hyp{}free definable. Moreover, they uniformly converge to
\(\vert\cdot\vert\).
\end{proof}

Let the theory \(\TPERF_p\) be the \(\Ldiv\)-theory containing \(\TVR\) and :
\begin{itemize}
\item \(p = 0\);
\item \(\phi : x\mapsto x^p\) is surjective;
\item there exists \(\varpi\neq 0\) such that \(\varpi^p = 0\) and for all
\(x\), if \(x^p = 0\), then \(\varpi \div x\) --- we say that \(\varpi\) is a
pseudo uniformizer. 
\end{itemize}

\begin{rem}
If \(M\models\PERF_{|p|=\alpha}\), with \(0 < \alpha < 1\) and
\(\varpi\in\Val(M)\setminus\{0\}\) is a topologically nilpotent divisor of
\(p\), then \(\Val(M)/(\varpi) \models \TPERF_p\).
\end{rem}

Let \(\Ldivpi = \Ldiv\cup\{\varpi_i\mid i\geq 1\}\) and let
\(\TPERF_{p,\varpi}\) be the \(\Ldivpi\)-theory \[\TPERF_p\cup \{\varpi_1\text{
is a pseudo uniformizer, }\varpi_{i+1}^p = \varpi_i \mid i\geq 1\}.\]

\begin{prop}
\label{tilt facts}
Let \(R\models\TPERF_{p,\varpi}\).
\begin{enumerate}
\item Let \(v^\flat : R^\flat \to \Gamma^\flat = \projlim_{x\mapsto px} \Gamma\)
the map induced by \(v\). Then \((R^\flat,v^\flat)\) is a characteristic \(p\)
perfect valued ring.

\item The map
\[\begin{array}{cccc}
\iota:&\Gamma&\to&\Gamma^\flat\\
&\gamma\neq\infty&\mapsto&(p^{-i}\gamma)_i\\
&\infty &\mapsto& \infty \end{array}\] is strictly increasing. The image of
\(\Gamma\setminus\{\infty\}\) by \(\iota\) is the initial segment \(<
pv(\varpi_1)\) of \(\Gamma^\flat\). Let \(\varpi^\flat\in R^\flat\) is the
element \((0,\varpi_1,\varpi_2,...)\). Using this identification, then
\(v^\flat(\varpi^\flat) = pv(\varpi_1)\).

\item The projection on the first co\hyp{}ordinate \({}^\natural : R^\flat \to R\)
induces an isomorphism \(R^\flat/(\varpi^\flat) \cong R\).

\item \(R^\flat\) is a \(\varpi^\flat\)-adically complete and separated valuation
ring.
\end{enumerate}
\end{prop}

\begin{proof}
\begin{enumerate}
\item It is routine to check that \(v^\flat\) is a truncated valuation. To show
that it is a valuation, it suffices to check that if \(\gamma,\delta
\in\Gamma^\flat\setminus\{\infty\}\), then \(\gamma + \delta \neq \infty\) ---
in that case, \(\Gamma^\flat \setminus\{\infty\}\) is the positive part of an
ordered abelian group and hence \(v^\flat\) is a valuation. Note first that
since \(\gamma = (\gamma_i)_{i\geq 0} \neq\infty\), we have \((p\gamma)_{i+1}
= \gamma_{i}\) and hence \(p\gamma \neq \infty\). As we may assume that
\(\delta\leq\gamma\), it follows that \(\gamma + \delta \leq 2\gamma \leq
p\gamma < \infty\). So \((R^\flat,v^\flat)\) is a valued ring and it is perfect of 
characteristic \(p\) by construction.

\item  For every \(\gamma\leq \delta\in\Gamma\), if \(p\gamma = p\delta\neq
\infty\), then \(p(\delta - \gamma) = 0\) and hence \(\delta = \gamma\). It
follows that \(\iota\) is well-defined. It is strictly increasing. The set
\(\iota(\Gamma\setminus\{\infty\})\) is the set of \(\gamma = (\gamma_i)_{i\geq
0} \in\Gamma^\flat\) such that \(\gamma_0 \neq \infty\), it is an initial
segment of \(\Gamma^\flat\). Moreover, if \(\gamma < pv(\varpi_1)\), then
\(\gamma_1 < v(\varpi_1)\) and hence \(\gamma_0 \neq \infty\).

With this identification, we have \(p^{-1}v^\flat(\varpi^\flat) =
v^\flat(\varpi^\flat)_1 = v(\varpi_1) < \infty\).

\item We have \((\varpi^\flat)^\natural = 0\). Conversely, if \(x = (x_i)_{i\geq
0} \in R^\flat\setminus\{0\}\) is such that \(x_0 = 0\), then, for some \(i >
0\), \(x_i \neq 0\). Since \(x_i^{p^i} = x_0 = 0\), it follows that
\(v(\varpi_1) \leq p^{i-1}v(x_i)\). So \(v^\flat(\varpi^\flat) = pv(\varpi_1)
\leq p^i v(x_i) = v^{\flat}(x)\) and the kernel of \({}^\natural\) is indeed
\((\varpi)\).

Now, the isomorphism \(R^\flat/(\varpi) \cong R^\flat/(\varpi^{p^n})\) given by
\(x\mapsto x^{p^n}\) induces an isomorphism \(R^\flat \cong \projlim_{x\mapsto
x^p} R^\flat/(\varpi) \cong \projlim R^\flat/(\varpi^n)\). Since the left to
right map can be checked to be the natural inclusion into the \(\varpi\)-adic
completion, it follows that \(R^\flat\) is \(\varpi\)-adically complete.

\item Let \(x = (x_i)_{i\geq 0},y = (y_i)_{i\geq 0}\in R^\flat\) be such that
\(v^\flat(x) \leq v^\flat(y)\). Taking \(p\)-th roots, we may assume that
\(v^\flat(y) < v^\flat(\varpi^\flat)\). Then, for all \(i\geq 0\), we have
\(v(x_i) \leq v(y_i) < \infty\) so we find some \(a_{i,0}\) such that \(y_i =
x_i a_{i,0}\). Choose some \(a_i\in R^\flat\) whose first co\hyp{}ordinate is
\(a_{i,0}\). Then \(y^{p^{-i}} \equiv x^{p^{-i}} a_i \mod \varpi^\flat\). It
follows that \(v^\flat(a_{i+1}^{p^{i+1}} - a_{i}^{p^{i}}) \geq p^i
v^\flat(\varpi^\flat)-v^\flat(x) \geq p^{i-1}v^\flat(\varpi^\flat)\). So \(a_i\)
is Cauchy and let \(a = \lim a_i\). Then \(y = \lim_i x a_i^{p^i} = x
a\).\qedhere
\end{enumerate}
\end{proof}

Let \(\alpha\in (0,1)\) and \(R\models \TPERF_{p,\varpi}\). Let
\(\vert\cdot\vert^\flat_\alpha : R^\flat \to [0,1]\) be the coarsening of \(v^\flat\)
normalized so that \(|\varpi^\flat|^\flat_\alpha = \alpha\). Then
\(R^\flat_\alpha = (R^\flat,\vert\cdot\vert^\flat_\alpha,\varpi^\flat)\) is a model of the
\(\LD\cup\{\varpi\}\)-theory \(\PERF_{|p|=0,|\varpi|=\alpha}\).

Let \(N\models \PERF_{|p|=0,|\varpi|=\alpha}\). The \(\Ldivpi\)-structure
\(N^\natural = (\Val(N)/(\varpi),\res_\varpi(\varpi_i)\mid i\geq 1)\) is a model
of \(\TPERF_{p,\varpi}\).

\begin{thm}
\label{tilt interp}    
Fix \(\alpha\in (0,1)\). Let \(R\models \TPERF_{p,\varpi}\) and
\(N\models\PERF_{|p|=0,|\varpi|=\alpha}\).

\begin{enumerate}
\item The set \(\Omega(R) := \{x\in R^{\mathbb{N}}
\mid\forall i\in{\mathbb{N}},\ x_{i+1}^p = x_i\}\) is quantifier\hyp{}free (continuously)
definable in \(R\).
\item The map \(\id_\Omega : \Omega(R) \to R^\flat\) is a quantifier\hyp{}free
interpretation of \(R^\flat_\alpha\) in \(R\).
\item The map \(\res_\varpi :
\Val(N) \to \Val(N)/(\varpi)\) is a quantifier\hyp{}free interpretation of
\(N^\natural\) in \(N\).
\item The maps \(\id_\Omega\) and \(\res_\varpi\) form a quantifier\hyp{}free
bi\hyp{}interpretation between \(\TPERF_{p,\varpi}\) and
\(\PERF_{|p|=0,|\varpi|=\alpha}\).
\end{enumerate}
\end{thm}

\begin{proof}
\begin{enumerate}
\item  The predicate \(d_n(x) = \inf_y \sup_{i\leq n} |x_i -
y^{p^{n-i}}|2^{-i}\) on \(R^{\mathbb{N}}\) uniformly converges to \(|x-\Omega|\) ---
indeed, the error is at most \(2^{-n-1}\). So \(\Omega\) is continuously
quantifier\hyp{}free definable.

\item Let \(\vert\cdot\vert_\alpha : R \to [0,1]\) be the norm such that \(|\varpi_1|_\alpha
= \alpha^{1/(p)}\) and for every \(x,y\in R\) let
\[D_\alpha(x,y) =
\left\{\begin{array}{ll}
0& \text{if }x \div y\\
|y|_\alpha&\text{otherwise.} \end{array}\right.\]

This is a (continuous) quantifier\hyp{}free predicate in \(R\). Let \(D^\flat\)
denote the interpretation of \(\rD\) in \(R^\flat_\alpha\). Then for every
\(x,y\in R^\flat\), we have \(D^\flat(x,y) = \lim_i D_\alpha(x_i,y_i)^{p^i}\)
which is quantifier\hyp{}free definable. Since \(|x|_\alpha^\flat =
D^\flat(0,x)\), the first item is proved --- as noted earlier,
by~\cite[Proposition~3.6]{BY-reconstruction}, the pullback of
\(\vert\cdot\vert_\alpha^\flat\) extends to a (pseudo)metric on
\(R^{\mathbb{N}}\).

\item For every \(x,y\in \Val(N)\), we have \(\res_\varpi(x) \div
\res_\varpi(y)\) if and only if \[\min\{\rD(\varpi,x),\rD(x,y)\} = 0,\] and
otherwise, we have \(\min\{\rD(x,y),\rD(\varpi,x)\} = |x|\). So \(\div\) is
induced by a quantifier\hyp{}free definable predicate and the second item follows.

\item The double interpretation of \(R\) in itself has domain \(\Omega(R)\) and
the projection on the first co\hyp{}ordinate \({}^\natural:\Omega(R) \to R\), which is
a term, induces an isomorphism of \(R^\flat/(\varpi)\) with \(R\).

The domain of the double interpretation of \(N\) in itself is \(X = \{x\in
\Val^{\mathbb{N}}\mid x_{i+1}^p \equiv x_i \mod \varpi\}\). By \cref{crucial}, the term
\(x\mapsto \lim_i x_i^{p^i}\) induces an isomorphism \((\Val(N)/(\varpi))^\flat
\to \Val(N)^\flat\to \Val(N)\).
\qedhere
\end{enumerate}
\end{proof}

Recall that if \(R\) is a ring, then \(\W_n(R)\) is a ring with underlying set
\(R^n\) where addition, subtraction and multiplication are given by polynomials
over \(\Zz\). The projection \(\W_{n+1}(R) \to \W_n(R)\) is a ring morphism ---
in other words the \(n\)-th co\hyp{}ordinate of addition, subtraction and
multiplication are computed using only the first \(n\) co\hyp{}ordinates --- and the
limit of this projective system is denoted \(\W(R)\). For every \(n\), \(\w_n(x)
= \sum_{i<n} x_i^{p^{n-i}} p^i\) defines a ring morphism \(\W_n(R) \to R\). We
write \([\cdot] : R\to \W(R)\) for the map sending \(a\) to \((a,0,0,\ldots)\); it
is a multiplicative section of the projection \(\W(R) \to R\) on the first
co\hyp{}ordinate. The \(\W_n\) --- and hence \(\W\) --- are functorial and for any
ring morphism \(f : R_1 \to R_2\), \(\W_n(f) : x\mapsto (f(x_i))_{i<n}\) is a
ring morphism. By construction, \(\W(R)\) is \(p\)-adically complete and
separated. Note that if \(R\) has characteristic \(p\) and the Frobenius
morphism is bijective on \(R\), then, for any \(a\in \W(R)\), we have \(a =
\sum_i [a_i]^{p^{-i}} p^i\).

Let now \(N\models\PERF_{|p|=0}\). We fix \(\varpi\in\Val(N)\) such that
\(0<|\xi| < 1\) and \(b\in \W(\Val(N))^\times\). Note that, \(b\in
\W(\Val(N))^\times\) if and only if \(b_0\in\Val(N)^\times\). Let \(\xi =
[\varpi] - pb \in \W(\Val(N))\).

\begin{lem}
For any \(x\in \W(\Val(N))\), there exists \(y\in \Val(N)\) and \(a\in
\W(\Val(N))^\times\) such that \(x \equiv [y]a \mod \xi\). 
\end{lem}

\begin{proof}
Let us first assume that \(v(x_0) < v(\varpi)\). Let \(x'' = (x - [x_0] -
[x_1]^{1/(p)}p)p^{-2}\) and \(a = \varpi x_0^{-1}\). As \(p \equiv
[\varpi]b^{-1}\), we have \(x\equiv [x_0]([1+ ax_1] + p[a]x'')\). Note that
\(v(a)>0\), so \(1 + ax_1\in\Val(N)^\times\) and hence \(x[x_0]^{-1} = [1+ ax_1]
+ p[a]x'' \in W(\Val(N))^\times\).

Let us now assume that \(v(x_0)\geq v(\varpi)\). Let \(y_0 = x_0\varpi^{-1}\)
and \(x' = (x - [x_0])p^{-1}\). Then \(x\equiv [\varpi]([y_0] + b^{-1}x')\equiv
p([y_0]b + x')\). It follows by induction that if the lemma fails for \(x\),
then, for all \(n\), there exists \(y_n,z_n\in\W(\Val(m))\) such that \(x = \xi
y_n + p^n z_n\) --- in other words, \(x\) is in the \(p\)-adic closure
\(\bigcap_n (\xi) + (p^n)\) of \(\xi\). However, note that \(\xi(y_{n+1} -
y_{n}) = p^{n}(z_{n} - z_{n+1})\). Since \(\xi_0 = \varpi \neq 0\), it follows
that \(y_{n+1} - y_{n} \in (p^n)\). So the sequence \(y_n\) is \(p\)-adically
Cauchy. Let \(y = \lim_n y_n\). We have \(x = \lim_n \xi y_n + p^n z_n = \xi y
\in (\xi)\) and hence \(x \equiv [0]\mod \xi\).
\end{proof}

For every \(x\in \W(\Val(N))\) let \(y\) and \(a\) as in the lemma. We define
\(|x|_\xi = |y|\) and \(v_\xi(x) = v(y)\).

\begin{prop}
\label{untilt facts}
\begin{enumerate}
\item The map \(v_\xi\) induces a valuation on \(A_\xi = \W(\Val(N))/(\xi)\) and
\(A_\xi\) is its valuation ring.
\item The map \(\vert\cdot\vert_\xi\) is a pseudo-norm on \(\W(\Val(N))\) with kernel
\((\xi)\). The associated norm on \(A_\xi\) is the (norm associated to the) rank
one coarsening of \(v_\xi\).
\item the map \(\res_\xi\circ [\cdot] : \Val(N) \to A_\xi\) induces an isomorphism
\(\Val(N)/(\varpi) \to A_\xi/(p)\).
\item We have \((A_\xi,\vert\cdot\vert_\xi)\models \PERF_{|p|=|\varpi|}\).
\item Let \(D_\xi\) be the interpretation of \(\rD\) in \((A_\xi,\vert\cdot\vert_\xi)\). The map
\[\begin{array}{ccc}
\W(\Val(N))^2&\to & [0,1]\\
(x,y) &\mapsto& D_\xi(\res_\xi(x),\res_\xi(y)) \end{array}\] is quantifier\hyp{}free
\((\varpi,b)\)-definable in the \(\LD\)-structure \(N\).

\item We have an isomorphism 
\[\begin{array}{ccc}
(A_\xi)^\flat &\to& \Val(N)\\
(x_{i})_{i} &\mapsto& \lim_i x_{i,0}^{p^i}\\
(\res_\xi([x^{p^{-i}}]))_i&\mapsfrom&x. \end{array}\]
\end{enumerate}
\end{prop}

\begin{proof}
\begin{enumerate}
\item First of all, by construction, \(v_\xi(x) = \infty\) if and only if \(x\in
(\xi)\) and \(v_\xi(1) \neq \infty\). Moreover, if \(x,y\in \Val(N)\) and
\(a,b\in\W(\Val(N))^\times\), then \([x]a [y]b = [xy]ab\) and hence \(v_\xi\) is
multiplicative. Let us now assume that \(v_\xi(x) \leq v_\xi(y)\) and let \(z =
yx^{-1}\). Then \([x]a + [y]b = [x](a + [z]b)\equiv [x][s]c \mod \xi\) for some
\(s\in\Val(N)\) and \(c\in\W(\Val(N))^\times\). So \(v_\xi([x]a + [y]b)\geq
v([x]a)\) and \(v_\xi\) does induce a valuation on \(A_\xi\). Also \([y]b = [x][z]b
\in [x]a W(\Val(N))\). So \(A_\xi\) is indeed its valuation ring.

\item The second statement follows immediately from the definition and the fact
that \(\vert\cdot\vert\) is the norm associated to \(v\) on \(N\).

\item The kernel of the map \(\Val(N)\to A_\xi/(p)\) induced by \(\res_\xi\circ[\cdot]\) is the
set of \(x\in \Val(N)\) such that \(v(x)\geq v_\xi(p)\). Since \(p\equiv
[\varpi]b^{-1} \mod\xi\), we have \(v_\xi(p) = v(\varpi)\) and the kernel is
indeed \((\varpi)\).

\item So far, we have shown that \((A_\xi,\vert\cdot\vert)\models\MVF\) and that \(|p|_\xi =
|\varpi|\). Since \(|A_\xi|_\xi = |\Val(N)|\), it is dense in \([0,1]\) and
since \(A_\xi/(p) \cong \Val(N)/(\varpi)\) and \(\Val(N)\) is perfect, the
Frobenius is surjective. This concludes the proof.

\item For any \(x \in \W(\Val(N))\) such that \(x_0 = \varpi z\), we write \(x' =
p^{-1}(x-[x_0]) + b[z]\). Then \(\res_\xi(px') = \res_\xi(x - [x_0] + pb[z]) = \res_\xi(x -
[z]([\varpi] - pb)) = \res_\xi(x)\).

Let \(x,y \in \W(\Val(N))\). If \(x,y\nin(\varpi)\), then \(v_\xi(x) = v(x_0)\),
\(v_\xi(y) = v(y_0)\) and hence \(D_\xi(\res_\xi(x),\res_\xi(y)) = \rD(x_0,y_0)\).
If \(x_0\nin (\varpi)\) but \(y\in (\varpi)\), then
\(D_\xi(\res_\xi(x),\res_\xi(py')) = 0 = \rD(x_0,y_0)\) as \(v_\xi (x) = v(x_0) <
v(\varpi) = v_\xi(p)\leq v_\xi(py')\). If \(x_0\in(\varpi)\) but
\(y_0\nin\varpi\), then \(D_\xi(\res_\xi(x),\res_\xi(y)) =
D_\xi(\res_\xi(px'),\res_\xi(y)) = |y|_\xi = |y_0|\) as \(v_\xi(x) \geq v_\xi(p) >
v_\xi(y)\). Finally if \(x,y\in (\varpi)\), then \[D_\xi(\res_\xi(x),\res_\xi(y)) =
D_\xi(\res_\xi(px'),\res_\xi(py')) = |p|_\xi D_\xi(\res_\xi(x'),\res_\xi(y')).\]

It follows that the quantifier\hyp{}free predicates \(t_n\) defined by \(t_0(x,y) =
0\) and
\[t_{n+1}(x) =
\left\{\begin{array}{ll}
    \rD(x_0,y_0) & \text{if }x_0 \nin(\varpi)\\
    |y_0|& \text{if }x_0\in(\varpi)\text{ and }y_0\nin\varpi\\
    |\varpi| t_n(x',y')&\text{if }x_0,y_0\in(\varpi) \end{array}\right.\]
uniformly converges to \(D_\xi(\res_\xi(x),\res_\xi(y))\).

\item Consider the isomorphisms
\[A_\xi^\flat\cong (\Val(N)/(\varpi))^\flat\cong \Val(N)^\flat\cong
\Val(N),\] where the first one is induced by \(\res_p\), the second is described
in \cref{crucial} and the last one is induced by projection to the first
co\hyp{}ordinate, giving the formula above. The inverse can be checked to be given by
\(x \mapsto (x^{p^{-i}})_i \mapsto (\res_\varpi(x^{p^{-i}}))_i \mapsto
(\res_\xi([x^{p^{-i}}]))_i\).
\qedhere
\end{enumerate}
\end{proof}

Let \(M\models\PERF_{|p| = \alpha}\), where \(0 < \alpha < 1\), and let
\({}^\sharp : \Val(M)^\flat \to \cO(M)\) denote the projection on the first
co\hyp{}ordinate. Recall that, by \cref{crucial}, \(\Val(M)^\flat \cong
(\Val(M)/(p))^\flat\) is a model
of \(\PERF_{|p|=0}\) that we denote
\(M^\flat\) whose valuation and norm we denote \(v^\flat\) and \(\vert\cdot\vert^\flat\)
respectively. The identification of \(v(\Val(M)/(p))\) inside
\(v^\flat(\Val(M)^\flat)\) induces an isomorphism \(v(\Val(M)) =
v^\flat(\Val(M)^\flat)\). The normalization of \(\vert\cdot\vert^\flat\) is chosen so that,
if \(\varpi\in\Val(M)^\flat\) has valuation \(v(p)\), then \(|\varpi|^\flat =
\alpha\).

\begin{fact}
\label{fact theta}
\begin{enumerate}
\item Let \(A\) be a \(p\)-adically complete ring, let \(R\) be a characteristic
\(p\) ring with bijective Frobenius morphism, and \(f : R \to A/(p)\) be a ring
morphism. Then there exists a (unique) ring morphism \(g : \W(R) \to A\) such
that
\[\begin{tikzcd}
\W(R)\arrow[r,"g"]\arrow[d]& A\arrow[d]\\
R \arrow[r,"f" {below}]& A/(p)
\end{tikzcd}\]
commutes. If \(f\) lifts to a multiplicative morphism \(\overline{f} : R \to A\), then for every \(x\in \W(R)\), we have \[g(x) = \sum_i \overline{f}(x_i^{p^{-i}})p^i.\] 
\item For every \(x\in Val(M)^\flat\), we have \(v^\flat(x) = v(x^\sharp)\) and
\(|x|^\flat = |x^\sharp|\).

\item The map \({}^\sharp : \Val(M)^\flat\to \Val(M)\) induces an isomorphism
\(\Val(M)^\flat/(\varpi) \to \Val(M)/(p)\). Also, for every \(x\in \)
\item The map \(\theta : \W(\cO(M)^\flat) \to \cO(M)\) defined by \[\theta
:\sum_i [x_i] p^i\mapsto\sum_i x_i^\sharp p^i\] is a surjective ring
homomorphism with kernel \(([\varpi] - pb)\), for any \(\varpi\in\Val(M)^\flat\)
and \(b\in \W(\Val(M)^\flat)\) such that \(v(\varpi) = v(p)\) and \(\theta(b) =
\varpi^\sharp p^{-1}\) --- in particular, \(v(b_0) = 0\).
\end{enumerate}
\end{fact}

\begin{proof}
\begin{enumerate}
\item For every \(n\in \Zz_{>0}\) and \(x \in R\), we have \(\w_n(px) = \sum_i
x_i^{p^{-i}}p^n \in p^n A\). It follows that \(\w_n\) induces a well defined
ring morphism \(\w_n : \W_n(A/(p)) \to A/(p)^n\). Then
\[g_n := \w_n \circ
\W_n(f) \circ \W_n(\phi^{-n}) : \W_n(R) \to \W_n(R) \to \W_n(A/(p)) \to
A/(p)^n\] is a ring homomorphism and for every \(x\in\W_n(R)\), we have \(g_n(x)
= \sum_i \overline{x_i^{p^{-n}}}^{p^{n-i}} p^i\), where
\(\overline{f(x_i^{p^{-n}})}\in A/(p)^n\) is any lift of \(f(x_i^{p^{-n}})\in
A\). Since \(\overline{f(x_i^{p^{n+1}})}^p\) is a lift of \(f(x_i^{p^{-n}})\in
A\), it follows that the \(g_n\) are compatible with the projections and hence
yield a ring morphism \(g : \W(R) \to \varprojlim_n A/(p)^n \cong A\). The
formula for \(g(x)\) when \(A \to A/(p)\) admits a multiplicative lifting
follows from the formulas for \(g_n\).

\item Recall that if \(x\in(\Val(M)/(p))^\flat\) is such that \(x_i\neq 0\),
then \(v^\flat(x) = p^{i}v(x_i)\). It follows that if \(x\in \Val(M)^\flat\),
\(v^{\flat}(x) = p^{i}v(x_i) = v(x_0) = v(x^\sharp)\), for \(i \gg 1\). As
\(|\varpi|^\flat =\alpha = |p|\) when \(v^\flat(\varpi) = v(p)\), the
normalizations match and we have \(|x|^\flat = |x^\sharp|\).

\item The map induced by \({}^\sharp\) coincides with the isomorphism
\[\Val(M)^\flat/(\varpi) \cong (\Val(M)/(p))^\flat/(\varpi) \cong
\Val(M)/(p)\] --- the last isomorphisms is described in \cref{tilt facts}.(3).

\item Applying the first statement to the ring morphism \({}^\sharp :
\Val(M)^\flat \to \Val(M) \to \Val(M)/(p)\) yields a ring morphism \(\theta :
\W(\Val(M)^\flat) \to \Val(M)\) which is exactly given by \(x = \sum_i [x_i] p^i
\mapsto \sum_i (x_i)^\sharp p^i\).

By construction, \([\varpi] + pb \in\ker(\theta)\) and \(v(b_0) = v(\theta(b)) =
v(\pi) - v(p) = 0\). There remains to prove that \(\overline{\theta} :
\W(\Val(M)^\flat)/([\varpi] + pb \in\ker(\theta)) \to \Val(M)\) is an
isomorphism. Note that the map \(\W(\Val(M)^\flat)/([\varpi] + pb,p)\cong
\Val(M)^\flat/(\varpi) \to \Val(M)/(p)\), induced by \(\overline{\theta}\)
modulo \(p\), is the map induced by \({}^\sharp\) which is an isomorphism.

Since \(p\) is not a zero divisor in \(\Val(M)\), it now follows by induction
that \(\overline{\theta}\) is an isomorphism modulo \(p^n\) for all \(n>0\); and
hence, since \(\W(\Val(M)^\flat)\) is \(p\)-adically complete,
\(\overline{\theta}\) is an isomorphism.\qedhere
\end{enumerate}
\end{proof}

\begin{lem}
Let \(M\models \PERF_{|p|=\alpha}\), where \(0\leq \alpha < 1\). The set
\(\Omega(M) := \{x\in \Val(M)^{\mathbb{N}} \mid\forall i\in{\mathbb{N}},\ x_{i+1}^p = x_i\}\)
is quantifier\hyp{}free definable.
\end{lem}

\begin{proof}
The definable predicates \(d_n(x) := \inf_{y\in\Val/(p)}\sup_{i\leq n} |x_i -
y^{p^{n-i}}|2^{-i}\) converge uniformly to \(d(x,\Omega)\). Indeed for any
\(y\in \Val(M)\), there exists \(z\in \Omega(M)\) such that \(y\equiv z\mod p\)
in which case \(y^{p^{n-i}}\equiv z^{p^{n-i}} \mod p^{n-i}\). It follows that
\(|x-\Omega| \leq d_n(x) \leq |x-\Omega| + 2(\min\{|p|,2^{-1}\})^{n}\).
\end{proof}

If \(a_i\in \Omega(M)\), we write \(a^\flat = (a_i)_i \in \Val(M)^\flat\).

\begin{rem}
\label{equiv norm tilt}
The norm on \(\Omega\) and the pullback of the norm from \(\Val(M)^\flat\) are
equivalent. Indeed the map \(x^\flat \to x_n\) induces an isomorphism
\(\Val(M)^\flat/(\varpi^{p^n})\cong \Val(M)/(p)\) so, \(x^\flat\equiv
y^\flat\mod \varpi^{p^n}\) if and only if \(x_n \equiv y_n \mod p\) --- in which
case \(x_{n-i} \equiv y_{n-i} \mod p^{i+1}\).
\end{rem}

Fix some \(a\in (0,1)\). Let \(\PERF_{|p|=0,|\xi| = \alpha}\) be the
\(\LD\cup\{\varpi_i,b_{i,j}\mid i,j\geq 0\}\)-theory
\[\PERF_{|p|=0}\cup\{|\varpi_{0}| = \alpha, b_{0,0}\in \Val^\times,
\varpi\in\Omega,b_j\in\Omega \mid j\geq 0\}.\] It might seem odd to name
\(p\)-th roots in characteristic \(p\), but we have to if we aim for 
quantifier\hyp{}free interpretations as the inverse Frobenius is not a term in the
language. To alleviate notation, we confuse \(\varpi\) with \(\varpi_0\) and
\(b\) with \(\sum_j [b_{j,0}] p^j\). As before, we \(\xi = [\varpi] - pb\).

Let \(\PERF_{|p|=\alpha,\xi}\) be the \(\LD\cup\{\varpi_i,b_{i,j}\mid i,j\geq
0\}\)-theory \[\PERF_{|p|=0}\cup\{\ker(\theta) = ([\varpi^\flat] - p \sum_{j\geq
0} [b_j^\flat]p^j)\}.\] If \(M\models \PERF_{|p|=\alpha,\xi}\), then \(M^\flat =
(\Val(M)^\flat,(\varpi_i^\flat)^{p^{-i}},(b_{j}^\flat)^{p^{-i}} \mid i,j\geq
0)\) is a model of \(\PERF_{|p|=0,|\xi|=\alpha}\) by \cref{fact theta}.(4).

\begin{thm}
\label{bi-int tilt}
Fix \(\alpha\in(0,1)\). Let \(M \models \PERF_{|p|=\alpha,\xi}\) and
\(N\models \PERF_{|p|=0,|\xi| = \alpha}\).
\begin{enumerate}
\item The map \(\id_\Omega : \Omega(M) \to \Val(M)^\flat\) is a quantifier\hyp{}free interpretation of \(M^\flat\) in \(M\).
\item Let \(f:\Val(N) \to (A_\xi)^\flat\) be the isomorphism of \cref{untilt
facts}.(6). Then \(N_\xi = (A_\xi,\vert\cdot\vert_\xi,f(\varpi),\W(f(b)))\) is a model
of \(\PERF_{|p|=\alpha,\xi}\).
\item The map \(\res_\xi : \W(\Val(N)) \to A_\xi\) is a quantifier\hyp{}free
interpretation of \(N_\xi\) in \(N\).
\item The maps \(\id_\Omega\) and \(\res_\xi\) are a quantifier\hyp{}free
bi\hyp{}interpretation between the theories \(\PERF_{|p|=\alpha,\xi}\) and
\(\PERF_{|p|=0,|\xi|=\alpha}\).
\end{enumerate}    
\end{thm}

\begin{proof}
\begin{enumerate}
\item Note that the map \({}^\sharp : \Omega \to \Val\) given by the first
co\hyp{}ordinate is a term and addition on \(\Omega\) which is given by \(\lim_i
(x_i + y_i)^{p^i}\) is a uniform limit of terms. Also, if \(D^\flat\) is the interpretation of \(\rD\) in
\(M^\flat_\xi\), then \(D^\flat(x,y) = \rD(x^\sharp,y^\sharp)\) is quantifier
free definable. The second statement follows.

\item The following diagram
\[\begin{tikzcd}
\W(\Val(N))\arrow[r,"\res_\xi"]\arrow[d,"\W(f)" {left}]& A_\xi\\
W(A_\xi^\flat) \arrow[ru,"\theta" {below}]
\end{tikzcd}\]
commutes: if \(x = (x_j)_j \in \W(\Val(N))\), then \(\W(f)(x) =
(\res_\xi([x_j^{p^{-i}}]))_{i,j}\) and hence \(\theta(\W(f)(x)) = \sum_j
\res_\xi([x_j^{p^{-j}}]) p^j = \res_\xi(x)\). So the kernel of \(\theta\) is generated by \(\W(f)(\xi) = [f(\varpi_0)] + p \W(f)(b)\).

\item By \cref{untilt facts}.(5), the interpretation of \(\rD\) in \(N_\xi\) is
induced by a quantifier\hyp{}free definable predicate in \(N\) --- and so is \(\vert\cdot\vert\)
since \(|x| = \rD(x,0)\). Since the map \(f\) is induced by \(x \mapsto
(x^{p^{-i}})_i\) it is a term on \(\varpi,b_j\in \Omega\).

\item The double interpretation of \(M\) in itself is
\(\W(\Val(M)^\flat)/(\xi)\). The map \(\theta\), which is induced by a term as
seen in its definition \cref{fact theta}.(4), induces an isomorphism with
\(\Val(M)\).

Conversely, the double interpretation of \(N\) in itself is isomorphic to
\(\Val(N)\) via the map of \cref{untilt facts}.(6) which is induced by a
term.
\qedhere
\end{enumerate}
\end{proof}

Let \(\TPERF_{p,\xi}\) be the \(\Ldivpi\cup\{b_{i,j}\mid i,j\geq 0\}\)-theory
\[\TPERF_{p,\varpi}\cup\{b_{0,0} \in R^\times, b_{i+1,j}^p = b_{i,j}\mid i,j
\geq 0\}.\] If \(M\models \PERF_{|p|=\alpha,\xi}\), then \(M^\natural =
(\Val(M)/(p),\res_p(\varpi_{i}),\res_p(b_{i,j})\mid i,j\geq 0)\) is a model of
\(\TPERF_{p,\xi}\). 

Let \(\alpha\in(0,1)\) and \(R\models \TPERF_{p,\xi}\). Then \(R^\flat_\alpha =
(R^\flat,\vert\cdot\vert_\alpha,(\varpi^\flat)^{p^{-i}},(b_j^{\flat})^{p^{-i}}\mid
i,j\geq 0)\) is a model of \(\PERF_{|p|=0,|\xi|=\alpha}\), where
\(\vert\cdot\vert_\alpha\) is normalized so that \(\varpi^\flat =
(0,\varpi_1,\varpi_2,\ldots)\) has norm \(\alpha\).

\begin{cor}
Fix \(\alpha\in(0,1)\). Let \(M\models \PERF_{|p|=\alpha,\xi}\) and let
\(R\models \TPERF_{p,\xi}\).
\begin{enumerate}
\item The map \(\res_p : \Val(M) \to \Val(M)/(p)\) is a quantifier\hyp{}free
interpretation of \(M^\natural\) in \(M\).
\item It forms a quantifier\hyp{}free bi\hyp{}interpretation with the composition \(\res_\xi \circ \id_\Omega\).
\end{enumerate}
\end{cor}

\begin{proof}
\begin{enumerate}
\item As in \cref{tilt interp}.(3), this follows from the fact that the
predicate \(\res_p(x)\div\res_p(y)\) is quantifier\hyp{}free definable.
\item The isomorphism \(W(R^\flat)/(\xi,p) \cong R^\flat/(\varpi) \cong R\) is
induced by the projection on the first co\hyp{}ordinate. Conversely, the map
\({}^\sharp : (\Val(M)/(p))^\flat \cong \Val(M)^\flat \to \Val(M)\) is induced
by the uniform limit of terms \(\lim_i x_i^{p^i}\) and so the isomorphism \(\theta :
\W((\Val(M)/(p))^\flat)/(\xi) \cong \W(\Val(M)^\flat)/(\xi) \cong \Val(M)\) is also induced by a uniform limit of terms.\qedhere
\end{enumerate}
\end{proof}

In conclusion, we have the following quantifier\hyp{}free bi\hyp{}interpretations.

\[\begin{tikzcd}
\PERF_{|p|=\alpha,\xi}
\ar[rr,shift left,"\flat"]
\ar[dr,"\natural" {left,outer sep=5pt}]
&&\PERF_{|p|=0,|\xi|=\alpha}
\ar[ll,shift left,"._\xi"]
\ar[ld,shift left,"\natural"]\\
&\TPERF_{p,\xi}\ar[ur,shift left,"\flat"]
\end{tikzcd}\]

\begin{rem}
\label{canon bi-int}
Note that the parameters here are not canonical, but the interpretation above
also induce bi\hyp{}interpretations (not necessarily quantifier\hyp{}free) between
\(\PERF_{|p|=\alpha}\), \(\TPERF_{p,(\xi)}\) and
\(\PERF_{|p|=0,|(\xi)|=\alpha}\) where, in the two last theories, we name the
ideal generated by \(\xi\).

Starting with a model \(N\) of \(\PERF_{|p|=0}\) we thus get a uniform family of interpretations parametrized by the set \(Y^{1} = \{([\varpi] + pb) \mid 0<|\varpi|<1\text{ and }b_0\in \Val(N)^\times\}\). Note that this is a type-definable subset of the interpretable set \(\W(\Val(N))/\W(\Val(N))^\times\) --- equipped with the norm \(|x - y\W(\Val(N))^\times| = |x_0 - y_0\Val(N)^\times|\).

If we fix some \(\varpi \in \Val(N)\) with \(0<|\varpi|<1\), we can consider \(X^1 = \{\xi\in Y^1 \mid v(\varpi) \leq v(\xi_0) < p v(\varpi)\}\). This is a \(\varpi\)-interpretable set whose points parametrize all untilts up to the the identification induced by the Frobenius automorphism on \(N\). This set corresponds to the degree one points on the Fargues-Fontaine curve.
\end{rem}

\begin{rem}
\label{rem:BY-MVF}
If we consider Ben Yaacov's theory of metric valued fields \cite{BY-MVF}, then
the tilt is also interpretable. However, it is unclear if we can interpret
untilts, or even axiomatize perfectoid fields, as the valuation ring is not
definable.  
\end{rem}

\section{Consequences}
\label{sec:consequences}
That the tilt-untilt correspondence is mediated by a quantifier\hyp{}free
bi\hyp{}interpretation gives fundamental properties of this correspondence as formal
consequences. 

\subsection{Tilting elementary equivalence}

The tilt, being a bi\hyp{}interpretation, preserves elementary equivalence as
continuous \(\Lrg\)-structures. However, \cref{Ldiv equiv Lrg} allows us to
recover results on the discrete \(\Ldiv\)-structure.

\begin{prop}
Fix \(\alpha\in(0,1)\) and \(M,N \models \PERF_{|p| = \alpha}\) have henselian valuation and bounded regular rank valuation group --- for example the valuation is rank one. Let \((\xi_M)
\subseteq \W(\Val(M^\flat))\) (resp. \((\xi_N)\subseteq \W(\Val(N^\flat))\))
denote the kernel of \(\theta\). Then \(M\) and \(N\) are
\(\Ldiv\)-elementarily equivalent if and only of \((M^\flat,(\xi_M))\) and
\((N^\flat,(\xi_N))\) also are.
\end{prop}

\begin{proof}
If \(M\) and \(N\) are \(\Ldiv\)-elementarily equivalent, then they are
\(\Lrg\)-elementarily equivalent (see \cref{Ldiv equiv Lrg}.2). As tilting is a
bi\hyp{}interpretation (see \cref{canon bi-int}), it follows that
\((M^\flat,(\xi_M))\) and \((N^\flat,(\xi_N))\) are \(\Lrg\)-elementarily
equivalent. By \cref{Ldiv equiv Lrg}.1, they also are \(\Ldiv\)-elementarily
equivalent --- note that the ideal \((\xi_M)\) is principal, so choosing
generators with the same type, we see that they remain elementarily equivalent
when the ideal is named.

Conversely, if \((M^\flat,(\xi_M))\) and \((N^\flat,(\xi_N))\) are
\(\Ldiv\)-elementarily equivalent, then they are \(\Lrg\)-elementarily
equivalent --- note that \(v(\xi_{M,0})\) (resp. \(v(\xi_{N,0})\)) is well
defined, allowing us to interpret the \(\Lrg\)-structure in the
\(\Ldiv\)-structure. As before, by bi\hyp{}interpretation, \(M\) and \(N\) are
\(\Lrg\)-elementarily equivalent and hence they are \(\Ldiv\)-elementarily
equivalent.
\end{proof}

\subsection{Fontaine\hyp{}Wintenberger isomorphism}
\label{sec:FW}

What goes by the name of the
Fontaine\hyp{}Wintenberger theorem 
is (an elaboration of) the assertion  
that for \( K \) a perfectoid
field of mixed characteristic,
there is a canonical isomorphism 
between the absolute Galois
group of \( K \) and 
the absolute Galois group of
its tilt \( K^\flat \).  Let us
note here
that this theorem follows 
formally from the correspondence being a quantifier\hyp{}free bi\hyp{}interpretation (\cref{bi-int tilt}).

Let us start by noting that tilting 
and untilting transform algebraically 
closed fields to algebraically closed
fields. This result is well-known.  Our point
in presenting it is to show that the explicit
computations based on Krasner's Lemma which 
usually go into its proof may be replaced
with the observation that our tilt/untilt 
bi-interpretation preserves existential closedness.

\begin{prop}
\label{prop:macvf-tilts}
If \(M \models \ACMVF_{0,p} \), then 
\(M^\flat \models \ACMVF \). 
Likewise, if \(N \models \ACMVF_p \) and we choose 
\(\xi\) and \(\alpha\) so that
\((N,\xi) \models \PERF_{|p|=0,|\xi|=\alpha}\), 
then \(N_\xi \models \ACMVF \).
\end{prop}
\begin{proof}
By 
Corollary~\ref{cor:binterp-ec}, 
\(M^\flat\) and \(N_\xi\) are existentially closed 
relative to \(\PERF_{|p|=0}\) and \(\PERF_{|p|=\alpha}\) respectively.
By~\cite[Theorem 2.4]{BY-MVF}, the 
models of \(\ACMVF\) are precisely the 
existentially closed structures
relative to \(\MVF\). As every model of \(\MVF\) has a perfectoid extension,  
it follows from
Lemma~\ref{lem:cotheory-ec},
that \(M^\flat\) and \(N_\xi\) are existentially closed 
relative to \(\MVF\).  Hence, 
\(M^\flat \models \ACMVF\)
and \(N_\xi \models \ACMVF\).  
That is, they are both algebraically
closed.
\end{proof}

As pointed out in \cite{FarFon}, the Fontaine\hyp{}Wintenberger theorem 
now follows using standard ideas in Galois theory. For the sake of 
completeness, we reproduce those 
here.

As is standard, 
for a field \(K\) we write 
\(K^\mathrm{alg}\) for 
its algebraic closure and 
if \(K \subseteq L\) is a 
subfield of a complete field 
\(L\), then we write 
\(\widehat{K}\) for 
the completion of 
\(K\) realized as a subfield 
of \(L\).

\begin{lem}
\label{lem:tilt-algebraic-closure}
If \( M \models \PERF_{|p| = \alpha} \) 
with \( \alpha > 0 \), then 
\( (\widehat{\alg{M}})^\flat = \widehat{\alg{(M^\flat)}} \).  
If  \( N \models \PERF_{|p| = 0,|\xi|=\alpha} \), then 
\( \widehat{\alg{N}}_\xi = \widehat{\alg{(N_\xi)}} \).
\end{lem}
\begin{proof}
By Proposition~\ref{prop:macvf-tilts},
\(\widehat{\alg{(M^\flat)}}\) is 
(or may be realized as) a 
subfield of \((\widehat{\alg{M}})^\flat\).
Since tilting is part of a bi\hyp{}interpretation between \( \PERF_{|p|=\alpha,\xi}\) and \(\PERF_{|p|=0,|\xi|=\alpha}\), there is some perfectoid
substructure \(M' \subseteq \widehat{\alg{M}} \) with 
\( (M')^\flat = \widehat{(M^\flat)^\mathrm{alg}} \)
and \( M' = (\widehat{(M^\flat)^\mathrm{alg}})_\xi \).
Applying Proposition~\ref{prop:macvf-tilts} again, \( M' \models \ACMVF \).
Hence, as \(\widehat{\alg{M}}\) 
is the smallest with respect to inclusion complete algebraically closed subfield of itself containing \(M\), 
\( M' = \widehat{\alg{M}}\), 
from which we conclude that 
\((\widehat{\alg{M}})^\flat = 
\widehat{(M^\flat)^\mathrm{alg}}\).

The argument for the untilt is 
essentially the same.
\end{proof}

\begin{rem}
Kedlaya and Temkin show in~\cite{Kedlaya-Temkin} that if \(k\) is a field
positive characteristic and \(K := \widehat{\alg{k((t))}}\) is the completion of
the algebraic closure of the field of Laurent series over \(k\), then there are
non-surjective continuous field endomorphisms of \(K\).  Shahoseini and
Pasandideh observe with~\cite[Example 29]{Shahoseini-Pasandieh-intro-perfectoid}
that it follows from this theorem of Kedlaya and Temkin that there is a
complete algebraically closed subfield \(L\) of \(\mathbb{C}_p^\flat = K\),
where \( \mathbb{C}_p\) is the completion of an algebraic closure of the field
\(\mathbb{Q}_p\) of \(p\)-adic numbers, for which \(L\) is not the tilt of any
substructure of \(\mathbb{C}_p\).

Such a field cannot contain any parameter \(\xi\) used to realize \(\mathbb{C}_p
= (\mathbb{C}_p^\flat)_\xi \). For example, we can choose \(\xi = [t]- p\). In
that case, any complete algebraically closed subfield of \(K\) containing \(t\)
must indeed be equal to \(K\).
\end{rem}

\begin{cor}
\label{cor:Galois-to-Galois}
If \( M \models \PERF_{|p| = \alpha} \) 
with \( \alpha > 0 \), then tilting induces an 
isomorphism of topological groups between 
the absolute Galois group, \(\operatorname{Gal}(\alg{M}/M)\), of \(M\)
and the absolute Galois group, \(\operatorname{Gal}((M^\flat)^\mathrm{alg}/M^\flat)\),
of its tilt. 
\end{cor}
\begin{proof}
Let \(\xi\) be the Fargues\hyp{}Fontaine parameter for which \(M = (M^\flat)_\xi\).  

Since the tilt operation is functorial, we obtain a continuous map 

\[\operatorname{Aut}_{\operatorname{cont}}(\widehat{M}^\mathrm{alg}/M) 
\to \operatorname{Aut}_{\operatorname{cont}}((\widehat{M}^\mathrm{alg})^\flat/M^\flat)  =
\operatorname{Aut}_{\operatorname{cont}}(\widehat{(M^\flat)^\mathrm{alg}}/M^\flat) \]

where the final equality comes from Lemma~\ref{lem:tilt-algebraic-closure}.

As for any complete field \(K\) every field automorphism of an algebraic extension which leaves \(K\) 
fixed pointwise is continuous, we have a natural identification 
\[\operatorname{Gal}(K^\mathrm{alg}/K) = \operatorname{Aut}_{\operatorname{cont}}(\widehat{K}^\mathrm{alg}/K) \text{ .} \]
Applying this observation to \(M\) and \(M^\flat\), the map coming from the functoriality of tilting may be
seen as a continuous map 
\[ \operatorname{Gal}(\alg{M}/M) \to \operatorname{Gal}((M^\flat)^\mathrm{alg}/M^\flat) \text{ .}\]

Likewise, functoriality of untilting with the parameter \(\xi\) gives rise to a continuous map 
\[ \operatorname{Gal}((M^\flat)^\mathrm{alg}/M^\flat) \to \operatorname{Gal}(\alg{M}/M) \text{ .}\]

As tilting and untilting form a bi-interpretation, these two operations are inverses of each other and the 
the corresponding maps on automorphism groups are inverses.  That is, the displayed maps are isomorphisms.
\end{proof}

As a consequence of Corollary~\ref{cor:Galois-to-Galois} we see that finite extensions of perfectoid fields are 
perfectoid and tilting preserves degrees of field extensions.

\begin{cor}
\label{cor:finite-perfectoid}
If \(K \models \PERF_{|p|=\alpha} \) and 
\( L/K \) 
is a finite extension, 
the \(L \models \PERF_{|p|=\alpha}\) and 
\([L:K] = [L^\flat:K^\flat]\). 
\end{cor}
\begin{proof}
Fix an embedding \(L \subseteq K^\mathrm{alg}\) and the Fargues\hyp{}Fontaine parameter \(\xi\) for which 
\(K = (K^\flat)_\xi\).

Let \(\rho:\operatorname{Gal}(K^\mathrm{alg}/K) \to
\operatorname{Gal}((K^\flat)^\mathrm{alg}/K^\flat)\) be the isomorphism of
absolute Galois groups given by tilting.  Let \(L' :=
\widehat{(K^\flat)^\mathrm{alg}}^{\rho(\operatorname{Gal}(K^\mathrm{alg}/L))} \)
be the fixed field in the completed algebraic closure of \(K^\flat\) of the
image of the absolute Galoius group of \(L\) under the tilting isomorphism.  As
\([\operatorname{Gal}(K^\mathrm{alg}/K):\operatorname{Gal}(K^\mathrm{alg}/L)] =
[L:K] \) and \(\rho\) is an isomorphism, we have \([L':K^\flat] =
[\operatorname{Gal}((K^\flat)^\mathrm{alg}/K^\flat):\operatorname{Gal}((K^\flat)^\mathrm{alg}/L')]
=
[\operatorname{Gal}((K^\flat)^\mathrm{alg}/K^\flat):\rho(\operatorname{Gal}(K^\mathrm{alg}/L))]
= [L:K] < \infty \).  Since \(K^\flat\) is perfectoid of positive
characteristic, that is, perfect and complete, so is the finite extension
\(L'\), and thus also \( (L')_\xi\).  Using the isomorphism \(\rho^{-1}\) coming
from untilting with \(\xi\), we see that \( (L')_\xi\) is the fixed field of
\(\rho^{-1} (\operatorname{Gal}((K^\flat)^\mathrm{alg}/L')) =
\operatorname{Gal}(K^\mathrm{alg}/L)\).  Hence, \(L = (L')_\xi\) is perfectoid.
\end{proof}



\subsection{Approximation lemma}
\label{sec:approximation-lemma}

A key step in the proof of the 
weight monodromy conjecture 
for complete intersections 
in~\cite{Scholze} is an 
approximation lemma 
whereby the tilt of
a hypersurface over a 
perfectoid field of mixed
characteristic is approximated
by hypersurfaces over the 
tilt in the 
sense that for 
any $\epsilon > 0$ it
is possible to find 
a hypersurface contained 
in the $\epsilon$-tubular 
neighborhood of the tilt 
of the original hypersurface.
Our original motivation
in pursuing this project was to 
extend the approximation lemma 
to all algebraic varieties, 
though we were not able
to achieve that aim.  Instead,
we have 
a weaker approximation theorem 
whereby tilts of (quantifier\hyp{}free)
\(\LD\)-definable sets 
may be approximated
by (quantifier\hyp{}free) \(\Ldiv\)-definable 
sets.

\begin{prop}
\label{prop:approximation-predicate}
Let \(M\models \PERF_{|p|=\alpha} \) be a perfectoid field of mixed
characteristic, let \( y \) be a tuple of variables, and let \(X \subseteq
\Omega(M)^y\) be an \(\LD(M)\)-definable subset. Define \(X^\flat\) to be \(\{
x^\flat \in (\Val(M)^\flat)^y : x \in X \} \). Then \(X^\flat\) is
\(\LD(M^\flat)\)-definable.

Moreover, for every \(\gamma \in v(\Val(M))\setminus\{\infty\}\), the set
\(X^\flat_\gamma = \{x\in\Val(M^\flat)^y\mid v(x-X^\flat)\geq \gamma\}\) is
\(\Ldiv(M^\flat)\)-definable.
\end{prop}
\begin{proof}
The tilting interpretation is part of a bi\hyp{}interpretation. By \cref{equiv
norm tilt}, the pullback of the norm on \(\Val^\flat\) is equivalent to the norm
on \(\Omega\), so, by Corollary~\ref{cor:image-definable}, the set \(X^\flat\)
is definable in \(K^\flat\).

The second part of the statement follows form the fact that \(X^\flat_\gamma\)
is the preimage of a definable set in \(\Val/\gamma\Val\) whose structure in
the \(\Ldiv(M)\)-induced structure.
\end{proof}

We can upgrade the statement of Proposition~\ref{prop:approximation-predicate}
so that if \(X\) is itself quantifier\hyp{}free definable, then the
first\hyp{}order formulas used in the approximation to \(X^\flat\) may be taken
to be quantifier\hyp{}free.

\begin{prop}
\label{prop:qf-approximation-predicate}
With the hypotheses and notation of
Proposition~\ref{prop:approximation-predicate}, if \(X\) is quantifier\hyp{}free
definable, then \(X^\flat\) is quantifier free \(\LD(M^\flat)\)-definable and
\(X^\flat_\gamma\) is \(\Ldiv(M^\flat)\)-definable.
\end{prop}

Given \cref{rem:uniform-limits}, the distance to \(X^\flat\) is therefore a
uniform limit of linear combinations of \(\Ldiv(M^\flat)\)-formulas.

\begin{proof}
Since \(X\) is quantifier\hyp{}free definable, we have that \(X(K) =
X(\widehat{K^\mathrm{alg}}) \cap K^y\). By Proposition~\ref{prop:macvf-tilts},
\((\widehat{K^\mathrm{alg}})^\flat\) is also algebraically closed. The proposition
now follows from quantifier elimination in, respectively, \(\ACMVF\) and
\(\ACVF\).
\end{proof}

\begin{rem}
If \(X \subseteq \Val(M)^y\) is a zero set of a (quantifier free)
\(\LD(M)\)-definable and for some \(\gamma\in v(\Val(M))\setminus\{\infty\}\),
we have \(X = X_\gamma = \{x\in \Val(M)^y \mid v(y - X) \geq \gamma\}\), then
for some \(\delta \in v(\Val(M))\setminus\{\infty\}\), we have \(X^\flat =
X^\flat_\delta\). So, by the above, \(X^\flat\) is itself (quantifier free)
\(\Ldiv(M^\flat)\)-definable.
\end{rem}

Let us now consider an affine variety \(V\) over \(M\models\ACMVF\) --- to be
precise, we consider an irreducible affine scheme of finite type over
\(\Val(M)\) with an \(\Val(M)\)-point.

\begin{lem}
The set \(V(\Val(M))\) is quantifier-free \(\LD(A)\)-definable.
\end{lem}

\begin{proof}
Let \((f_i)_{i\leq n}\in A[x]\) be equations defining \(V\). Let \(\alpha(x) =
\min_i(v(f_i(x)))\) and let \(d_V(x)\) be the \(\Ldiv(M)\)-definable function
\(x \mapsto \sup_{y\in V(\Val(M))} v(x-y)\) --- this map is well defined as
\(v(M)\) is definably complete.  We show that, for some \(n \in \Zz_{>0}\) and
\(\gamma\in v(\Val(M))\), we have \(\alpha(x) \leq \gamma + n d_V(x)\). Our argument 
is extracted from the 
proof of~\cite[Lemma 6.6]{Hr-ML}.  Our
lemma follows by applying to this result to the rank one coarsening of \(M\) and
\cite[Proposition~2.19]{continuous-logic-book}.

If the inequality fails, then in some (discrete) ultrapower \(M^\star\) of
\((M,v)\), we find \(x\) with coordinates in \(\Val(M^\star)\) such that, if
\(\Delta \leq v(M^\star)\) denotes the convex group generated by \(v(M)\) and
\(d_V(x)\), we have \(\alpha(x) > \Delta\). If \(w\) is the coarsening
associated to \(\Delta\) and \(\res_w : \Val\to k_w\) denotes the reduction
modulo the maximal ideal of \(w\), then \(\res_w(x) \in V(k_w)\). Then there
exists \(a \in V(\Val(M^\star))\) such that \(\res_w(a) = \res_w(x)\) (see
\cite[Theorem~3.2.4]{Hal-GenSt}). But then \(v(x-a) > \Delta\) and hence
\(d_X(x) > \Delta\), contradicting the definition of \(\Delta\).
\end{proof}

Considering the irreducible components of the set defined by \(x^{p^i} \in V\),
we see that this set is also quantifier free \(\LD(M)\)-definable. Taking a
uniform limit, we see that \(\{x\in \Omega\mid x_0 \in V\}\) is quantifier free
\(\LD(M)\)-definable. By \cref{prop:qf-approximation-predicate}, the set
\(V^\flat = \{x\in\Val^\flat \mid x^\sharp \in V\}\) is quantifier free
\(\LD(M^\flat)\)-definable. If we choose some \(\gamma\in
v(\Val(M))\setminus\{\infty\}\), then \(V^\flat_\gamma\) is
\(\Ldiv(M^\flat)\)-definable.

Ideally we would like \(V^\flat_\gamma\) to be the \(\gamma\)-neighborhood of
some variety \(W_\gamma\) over \(M^\flat\) of dimension equal to that of \( V\).
In~\cite{Scholze}, this is achieved through a “syntactic trick”: one may trace
through the translation given by tilting to see that for a hypersurface \(X\)
defined by \(f = 0\) where \( f \) is a polynomial, \( X^\flat \) is defined by
\(g = 0\) for some \(g\) such that \(g^\sharp\) is a good approximation of
\(f\). For a complete intersection, the approximating variety \( W \) is found
by approximating the tilts of a minimal set of defining equations.  The formulas
we obtain need not have this form.

\subsection{Type spaces}
\label{sec:type-spaces}

As we have seen 
with Propositions~\ref{prop:interp-type} and~\ref{prop:homeo-type-spaces},
interpretations
induce continuous maps between 
type spaces and for
an isometric bi\hyp{}interpretation
this induced map is
a homeomorphism.
Specializing to the case of the 
tilt/untilt correspondence, we 
may understand the
type spaces themselves as adic 
spaces, albeit considered with a 
finer
topology than usual, so
that corresponding 
homeomorphisms between 
type spaces may be seen 
as a kind of equivalence of
adic spaces over a perfectoid
field and its tilt.

Let \(M \models \PERF_{|p|=\alpha}\) be 
perfectoid and let \(A\leq \Val(M)\) be a perfectoid
substructure. For every type-\(A\)-definable set \(X \subseteq\Omega^n\), 
let \(X^\flat\subseteq (\Val^\flat)^n\) denote the corresponding type-\(A^\flat\)-definable set.

\begin{thm}
\label{homeo types tilt}
The spaces \(\TP_X(A)\) are \(\TP_{X^\flat}(A^\flat)\) are homeomorphic. Moreover, the homeomorphism is functorial and induces an equivalence of categories:
\[\left\{
\begin{array}{c}
\text{type-\(A\)-definable subsets}\\
\text{of (cartesian powers of) } \Omega(M)
\end{array}\right\}
\leftrightarrow \left\{
\begin{array}{c}
\text{type-\(A^\flat\)-definable subsets}\\
\text{of (cartesian powers of) } \Val(M^\flat)
\end{array}\right\}.\]
\end{thm}

\begin{proof}
This is a special case of Proposition~\ref{prop:homeo-type-spaces}.
\end{proof}

Let us now assume \(M\models \ACMVF\). Then we are considering quantifier-free
type spaces that are closely related to adic spaces.

In what follows, we write \(T\) for a (potentially infinite countable) tuple of
variables and \(A \langle T \rangle\) for the Tate algebra of convergent power
series over \(A\) in those variables. We allow valuations to have kernels.
Perhaps, the reader would prefer to call these ``pseudovaluations''.

\begin{fact}
\label{lem:equivalent-continuous-valuation}
Let $w$ be a valuation on \(A\langle T\rangle\) extending $v$. The following are equivalent:
\begin{enumerate}
    \item for every $\gamma\in w(A\langle T\rangle)$, the set $\{f\in A\langle
    T\rangle\mid w(f) \geq \gamma\}$ contains $aA\langle T\rangle$ for some
    $a\in A$.
    \item for every $f\in A\langle T\rangle$, if \(w f < \infty\), then there
    exists \(\gamma\in v(A)\) such that \(w f \leq \gamma\).
\end{enumerate}
\end{fact}

\begin{proof}
Indeed for any \(f\in A\langle T\rangle\) and any \(a\in A\), we have \(w(f)
\leq v(a)\) of and only if \(aA\langle T\rangle \subseteq \{g\in A\langle
T\rangle\mid w(g) \geq \gamma\}\).
\end{proof}

A valuation satisfying the equivalent conditions of
Lemma~\ref{lem:equivalent-continuous-valuation} is said to be continuous. Let
\(\Cont_v(A\langle T\rangle)\) be the set of continuous valuations on \(A\langle
T\rangle\) extending $v$, up to equivalence. We endow \(\Cont_v(A\langle
T\rangle)\) with the topology generated by the sets of the form \(\{w f\geq w g
\neq\infty\}\) for all $f,g\in A\langle T\rangle$. If \(v\) has finite rank on
\(A\), this space is exactly the fiber of the adic space \(\Spa(A\langle
T\rangle,A\langle T\rangle)\) over the point \(v\) of \(\Spa(A,A)\).

\begin{prop}
There is a continuous bijection \[S_{\Val^{|T|}}(A) \to \Cont_v(A\langle
T\rangle).\] The topology on $S_{V}(A)$ is the constructible topology
associated to $\Cont_v(A\langle T\rangle)$.
\end{prop}

\begin{proof}
Fix any \(p\in S_{\Val^{|T|}}(A)\), then we define a valuation \(v_p\) on \(A
\langle T \rangle\) by \(v_p(f) \leq v_p(g)\) whenever \(p(T)\models D(f,g) =
0\). Since \(vA^\times\) is cofinal in the valuation group of any elementary
extension of \(M\), \(v_p\) is continuous. Injectivity follows from quantifier
elimination (\cref{ACMVF EQ}). Conversely, if \(v\) is a continuous valuation on
\(A\langle T \rangle\), then tensoring by \(\Val(M)\), modding out by the kernel
and taking the valuation ring, we obtain an extension \(R\) of \(M\). As \(M\)
is existentially closed in \(R\), the type of \(T\) over \(A\) is an then
element of \(S_{\Val^{|T|}}(A)\) whose associated valuation is \(v\) itself.

For continuity, we need to prove that the pre-image of \[\{w \in
\Cont_v(A\langle T \rangle) : w f \geq w g \neq\infty\}\] is open. Note that,
for all non-zero \(a\in A\), the set \[\{ w \in \Cont_v(A\langle T \rangle) : w
f \geq w g \leq v a \}\] is clopen, as it factorizes through
some quotient modulo \(c\in A\). Since \[\{w \in \Cont_v(A\langle T \rangle) : w
f \geq w g \neq\infty\} = \bigcup_{a\in A^\times} \{w \in \Cont_v(A\langle T
\rangle) : w f \geq wg \leq v a \},\] this set is open. 

To prove the last statement, we have to prove that the open quasi-compact
subsets of \(\Cont_v(A\langle T\rangle)\) --- i.e. the sets \(\{w f_i \geq w g
\neq\infty \mid i<n\}\) where \(\fm\subseteq (f_i\mid i<n)\) --- is clopen. But
then some linear combination of the \(f_i\) is a non-zero constant and hence
\(\{w f_i \geq w g \neq\infty \mid i<n\} = \{w f_i \geq w g \leq v a \mid
i<n\}\) for some $a\in A^\times$, which is indeed clopen.
\end{proof}

The continuous bijection above is compatible with the projection to
\(\Spec(A\langle T\rangle)\). Thus, for every ideal \(I\subseteq A\langle
T\rangle\), it induces a continuous bijection \[S_{V}(A) \to \Cont_v(A\langle
T\rangle/I),\] where \(V\subseteq \Val^{|T|}\) is the zero set of all \(f\in
I\). It also naturally induces a continuous bijection on affinoid subspaces.

Note that this continuous bijection identifies the space \(S_{\Omega^n}(A)\) of
\cref{homeo types tilt} with the constructible space associated to
\(\Cont_v(A\langle T^{p^{-\infty}}\rangle)\).

\begin{rem}
\label{rem:correct-topology}
Recovering the right topology on type spaces seems to be a matter of working in
positive continuous logic; but the exact setup for this is not entirely obvious
to the authors.
\end{rem}

\bibliographystyle{siam}
\bibliography{tilteq}

\end{document}